\pdfoutput=1
\documentclass[11pt,oneside]{amsart}
\usepackage[
paper=a4paper,
headsep=20pt,text={136mm,208mm},centering,includehead
]{geometry}
\usepackage{graphicx}
\usepackage[dvipsnames,usenames]{xcolor}
\usepackage[colorlinks=true,urlcolor=ForestGreen,linkcolor=ForestGreen,citecolor=ForestGreen]{hyperref}
\hypersetup{
	linkbordercolor={1 0 0}, 
	citebordercolor={0 1 0} 
}
\usepackage{amssymb,mathtools,mathrsfs,stmaryrd,wasysym,bbm}
\usepackage[UKenglish]{babel}
\usepackage[noabbrev]{cleveref}
\usepackage{tikz}
\usetikzlibrary{trees,decorations.pathmorphing,decorations.markings,matrix,shapes}
\usepackage[backgroundcolor=green,linecolor=green,bordercolor=green]{todonotes}
\usepackage[shortlabels]{enumitem}
\usepackage{float}

\usepackage{libertine}
\usepackage[vvarbb,bigdelims,cmintegrals,libertine]{newtxmath}

\iftrue
\makeatletter
\def\@settitle{%
	\vspace*{-20pt}
	\begin{flushleft}%
		\baselineskip14\p@\relax
		\normalfont\bfseries\LARGE
		\@title
	\end{flushleft}%
}
\def\@setauthors{%
	\begingroup
	\def\thanks{\protect\thanks@warning}%
	\trivlist
	\large \@topsep30\p@\relax
	\advance\@topsep by -\baselineskip
	\item\relax
	\author@andify\authors
	\def\\{\protect\linebreak}%
	\authors
	\ifx\@empty\contribs
	\else
	,\penalty-3 \space \@setcontribs
	\@closetoccontribs
	\fi
	\normalfont
	\@setaddresses
	\endtrivlist
	\endgroup
}
\def\@setaddresses{\par
	\nobreak \begingroup\raggedright
	\small
	\def\author##1{\nobreak\addvspace\smallskipamount}%
	\def\\{\unskip, \ignorespaces}%
	\interlinepenalty\@M
	\def\address##1##2{\begingroup
		\par\addvspace\bigskipamount\noindent
		\@ifnotempty{##1}{(\ignorespaces##1\unskip) }%
		{\ignorespaces##2}\par\endgroup}%
	\def\curraddr##1##2{\begingroup
		\@ifnotempty{##2}{\nobreak\noindent\curraddrname
			\@ifnotempty{##1}{, \ignorespaces##1\unskip}\/:\space
			##2\par}\endgroup}%
	\def\email##1##2{\begingroup
		\@ifnotempty{##2}{\smallskip\nobreak\noindent E-mail address%
			\@ifnotempty{##1}{, \ignorespaces##1\unskip}\/:\space
			\ttfamily##2\par}\endgroup}%
	\def\urladdr##1##2{\begingroup
		\def~{\char`\~}%
		\@ifnotempty{##2}{\nobreak\noindent\urladdrname
			\@ifnotempty{##1}{, \ignorespaces##1\unskip}\/:\space
			\ttfamily##2\par}\endgroup}%
	\addresses
	\endgroup
	\global\let\addresses=\@empty
}
\def\@setabstracta{%
	\ifvoid\abstractbox
	\else
	\skip@25\p@ \advance\skip@-\lastskip
	\advance\skip@-\baselineskip \vskip\skip@
	\box\abstractbox
	\prevdepth\z@ 
	\vskip-15pt
	\fi
}
\renewenvironment{abstract}{%
	\ifx\maketitle\relax
	\ClassWarning{\@classname}{Abstract should precede
		\protect\maketitle\space in AMS document classes; reported}%
	\fi
	\global\setbox\abstractbox=\vtop \bgroup
	\normalfont\small
	\list{}{\labelwidth\z@
		\leftmargin0pc \rightmargin\leftmargin
		\listparindent\normalparindent \itemindent\z@
		\parsep\z@ \@plus\p@
		
	}%
	\item[\hskip\labelsep\bfseries\abstractname.]%
}{%
	\endlist\egroup
	\ifx\@setabstract\relax \@setabstracta \fi
}

\def\ps@headings{\ps@empty
	\def\@evenhead{%
		\setTrue{runhead}%
		\normalfont\scriptsize
		\rlap{\thepage}\hfill
		\def\thanks{\protect\thanks@warning}%
		\leftmark{}{}}%
	\def\@oddhead{%
		\setTrue{runhead}%
		\normalfont\scriptsize
		\def\thanks{\protect\thanks@warning}%
		\rightmark{}{}\hfill \llap{\thepage}}%
	\let\@mkboth\markboth
}\ps@headings

\def\section{\@startsection{section}{1}%
	\z@{-1.2\linespacing\@plus-.5\linespacing}{.8\linespacing}%
	{\normalfont\bfseries\Large}}
\def\subsection{\@startsection{subsection}{2}%
	\z@{-.8\linespacing\@plus-.3\linespacing}{.3\linespacing\@plus.2\linespacing}%
	{\normalfont\bfseries\large}}
\def\subsubsection{\@startsection{subsubsection}{3}%
	\z@{.7\linespacing\@plus.1\linespacing}{-1.5ex}%
	{\normalfont\bfseries}}
\def\@secnumfont{\bfseries}
\makeatother
\fi 

\makeatother

\newtheorem{theorem}{Theorem}[section]
\newtheorem*{theorem*}{Theorem}
\newtheorem*{corollary*}{}
\newtheorem{proposition}[theorem]{Proposition}
\newtheorem{corollary}[theorem]{Corollary}
\newtheorem{lemma}[theorem]{Lemma}

\newtheorem*{conjecture*}{Conjecture}

\newtheorem{question}{Question}

\newtheorem*{remark}{Remark}

\theoremstyle{definition}
\newtheorem{defin}[theorem]{Definition}
\newenvironment{definition}
{\pushQED{\qed}\defin}
{\popQED\enddefin}

\makeatletter

\providecommand{\proofname}{Proof}
\makeatother

\numberwithin{equation}{section}

\newcommand{\cdkh}{CDKh}
\newcommand{\dkh}{DKh}

\newcommand{\Q}{\mathbb{Q}}
\newcommand{\Z}{\mathbb{Z}}
\newcommand{\N}{\mathbb{N}}
\newcommand{\sg}{\mathfrak{s}}
\newcommand{\vup}{v^{\text{u}}_+}
\newcommand{\vum}{v^{\text{u}}_-}
\newcommand{\vlp}{v^{\ell}_+}

\newcommand{\dpl}{\overset{\bullet}{+}}
\newcommand{\dm}{\overset{\bullet}{-}}

\begin{document}
	
	\vspace*{-50pt}
	\title{Ascent concordance}
	
	\author{William Rushworth}
	\address{
		Department of Mathematics and Statistics, McMaster University
	}
	\email{\href{mailto:will.rushworth@math.mcmaster.ca}{will.rushworth@math.mcmaster.ca}}
	
	\def\subjclassname{\textup{2010} Mathematics Subject Classification}
	\expandafter\let\csname subjclassname@1991\endcsname=\subjclassname
	\expandafter\let\csname subjclassname@2000\endcsname=\subjclassname
	

\begin{abstract}
A cobordism between links in thickened surfaces consists of a surface \( S \) and a \(3\)-manifold \( M \), with \( S \) properly embedded in \( M \times I \). We show that there exist links in thickened surfaces such that if \( (S,M) \) is a cobordism between them in which \( S \) is simple, then \( M \) must be complex. That is, there are cases in which low complexity of the surface does not imply low complexity of the \(3\)-manifold.

Specifically, we show that there exist concordant links in thickened surfaces between which a concordance can only be realised by passing through thickenings of higher genus surfaces. We exhibit an infinite family of such links that are detected by an elementary method and other families of links that are not detectable in this way. We investigate an augmented version of Khovanov homology, and use it to detect these families. Such links provide counterexamples to an analogue of the Slice-Ribbon conjecture.
\end{abstract}

\maketitle

\section{Introduction}\label{1Sec:intro}
This paper is concerned with the relationship between different measures of complexity of link cobordisms. As described below, a cobordism between links in thickened surfaces consists of a surface \( S \) and a \(3\)-manifold \( M \), with \( S \) properly embedded in \( M \times I \). We show that there exist links in thickened surfaces such that if \( (S,M) \) is a cobordism between them in which \( S \) is simple, then \( M \) must be complex. That is, there are cases in which low complexity of the surface does not imply low complexity of the \(3\)-manifold.

We do so by exhibiting concordant links in thickened surfaces between which a concordance can only be realised by passing through thickenings of higher genus surfaces. We exhibit infinite families of such links. One family is detected by an elementary method, to which the other families are not amenable. We investigate the affect of various types of cobordisms on an augmented version of Khovanov homology due to Manturov and the author \cite{Rushworth2017b}, and show that it detects these families. These infinite families provide counterexamples to an analogue of the Slice-Ribbon conjecture for knots in \( S^3 \).

The complexity of a cobordism, \( S \), between two knots in \( S^3 \) may be measured along two axes: the complexity of \( S \) as an abstract surface, and that of the embedding \( S \hookrightarrow S^3 \times I \). The former is a measure of intrinsic complexity, the latter an extrinsic measure. A cobordism of minimal intrinsic complexity is an annulus, known as a \emph{concordance}. One way to measure extrinsic complexity is via Morse theory: with this measure \( S \) is of minimal extrinsic complexity if it possesses no index \( 2 \) Morse critical points (see \cite[Chapter \(3\)]{LivingstonNaik}, for example).

The Slice-Ribbon Conjecture posits that if there exists a cobordism of minimal intrinsic complexity from a knot to the unknot, then there exists a cobordism of minimal intrinsic and extrinsic complexity.
\begin{conjecture*}[Slice-Ribbon]
Let \( K \) be a knot in \( S^3 \). If there exists a concordance from \( K \) to the unknot, then there exists a concordance with no index \(2\) Morse critical points.
\end{conjecture*}

This paper is concerned with the complexity of cobordisms between links in \(3\)-manifolds other than \( S^3\). Specifically, we consider links in thickened surfaces: let \( \Sigma_g \) be a closed orientable surface of genus \( g \), and \( L_1 \hookrightarrow \Sigma_{g_1} \times I \) and  \( L_2 \hookrightarrow \Sigma_{g_2} \times I \) links in thickenings of \( \Sigma_{g_1} \) and \( \Sigma_{g_2} \). A concordance between \( L_1 \) and  \( L_2 \) is a pair \( \left( S, M \right) \), where \( M \) is a compact orientable \(3\)-manifold with \( \partial M = \Sigma_{g_1} \sqcup \Sigma_{g_2} \), and \( S \) a disjoint union of annuli properly embedded in \( M \times I \) such that each annulus has a boundary component in both \( L_1 \) and \( L_2 \) \cite{TuraevCob}.

In addition to analysing the surface \( S \), we may pose new questions about the \(3\)-manifold \( M \). In this new setting we must alter the definition of extrinsic complexity of a concordance, taking into account the complexity of the target \( M \times I \). The measure of extrinsic complexity splits into two distinct aspects:
\begin{enumerate}[(i)]
	\item the complexity of the target \( M \times I \)
	\item the complexity of the image of the embedding \( S \hookrightarrow M \times I \).
\end{enumerate}
The measure of intrinsic complexity and measure (ii) may be naturally carried over from the classical case. It remains to choose measure (i). In this paper we study the complexity of \( M \times I \) by considering the surfaces appearing as level sets of a Morse function on \( M \). A natural generalization of the Slice-Ribbon conjecture to links in thickened surfaces, therefore, states that if there exists an intrinsically simple cobordism between two links, then there exists a cobordism which is intrinsically simple and simple with respect to both extrinsic measures (i) and (ii). We exhibit counterexamples to this generalization, presenting concordant links in thickened surfaces that are not concordant via simple cobordisms with respect to measure (i).

Let us outline our choice of measure (i). Given a concordance \( ( S, M ) \), let \( f \) be a Morse function on \( M \). Up to isotopy we may assume that \( S \) is transverse to the \( I \) factor of \( M \times I \). Under the convention \( f \left( \Sigma_{g_1} \right) = \lbrace 1 \rbrace \) and \( f \left( \Sigma_{g_2} \right) = \lbrace 0 \rbrace \), when traversing the concordance from \( \Sigma_{g_1} \) to \( \Sigma_{g_2} \) an index \(2\) critical point of \( f \) is a \(3\)-dimensional \(1\)-handle addition. That is, the genus of level surfaces of \( f \) increases by \(1\) when passing an index \(2\) critical point. Postponing the precise definition until \Cref{2Subsec:cobordism}, we say that an index \(2\) critical point is \emph{exceeding} if the genus of level surfaces appearing immediately after it is greater than both \( g_1 \) and \(g_2\). Concordances for which \( M \) does not contain an exceeding index \( 2\) critical point are declared to be of minimal complexity with respect to measure (i), and are known as \emph{descent}. Therefore a concordance that is not descent must pass through surfaces of greater genus than both the initial and terminal surface.

Our main result establishes that there exist representatives of the same concordance class that cannot be seen to be concordant without introducing exceeding critical points. That is, we show that the existence of an intrinsically simple cobordism does not imply the existence of an intrinsically simple cobordism that is also simple with respect to measure (i).
\begin{theorem*}\label{1Thm:main}
There exist concordant links in thickened surfaces such that if \( \left( S, M \right) \) is a concordance between them, then \( M \) contains an exceeding index \(2\) critical point. Such links are said to be \emph{ascent concordant}. Indeed, there exist infinite families of ascent concordant links.
\end{theorem*}
This result follows from \Cref{4Thm:ascent}, and pairs of ascent concordant links are given in \Cref{4Fig:links1,4Fig:links2,4Fig:links3}. Note that the qualifier exceeding must be added to make the problem nontrivial: there are concordant links \( L_1 \hookrightarrow \Sigma_{g_1} \times I \) and  \( L_2 \hookrightarrow \Sigma_{g_2} \times I \) with \( g_1 < g_2 \), so that any concordance between them must contain an index \(2\) critical point (not necessarily exceeding).

Denote by \( J \) and \( J' \) the links depicted in \Cref{4Fig:links2}. As described below, we use an augmented version of Khovanov homology to prove that \( J \) and \( J' \) are ascent concordant. During this proof a concordance between them, \( (S,M) \), is described. In this concordance \( M \) is equal to \( \Sigma_1 \times I \) in a non-minimal handle decomposition: \( M \) contains a cancelling pair of handles. This in turn yields a non-minimal handle decomposition of \( M \times I \), into which \( S \) is properly embedded. As described on \cpageref{4Page:handles}, the result that \( J \) and \( J' \) are ascent concordant implies that the non-minimal handle decomposition of \( M \times I \) cannot be simplified in the complement of \( S \). Our main result therefore yields an application of a link homology theory to a problem of knotted surfaces. In particular, it is evidence that the augmented version of Khovanov homology contains subtle information regarding the complements of such surfaces.

Let \( \mathcal{L}_g \) denote the set of links in thickened surfaces of genus less than or equal to \( g \), and \( \mathcal{C}_g \) the quotient of \( \mathcal{L}_g \) obtained by identifying two links if they are concordant but not ascent concordant. Our main result demonstrates that \( \mathcal{C}_g \) is not a proper subset of \( \mathcal{C}_{g'} \), for \( g < g' \), in general. This is an instance of the ubiquitous phenomenon of `increase-before-decrease', as exhibited by classical knot diagrams \cite{Kauffman2012hard}, presentations of groups, and handle decompositions of manifolds.

In contrast, let \( K_1 \hookrightarrow \Sigma_{0} \times I \), \( K_2 \hookrightarrow \Sigma_{0} \times I \) be knots in the thickened \(2\)-sphere. A result of Boden and Nagel \cite{Boden2016} implies that if \( K_1 \) and \( K_2 \) are concordant, then they are not ascent concordant. We exhibit ascent concordant links with ambient space \( \Sigma_1 \times I \), so that the ascent phenomenon is seen to occur as soon as one passes to thickened surfaces of nonzero genus.

One way to approach the Slice-Ribbon Conjecture is to attempt to produce invariants of ribbon concordance that are not invariant under generic concordance. The case of ascent concordance is similar. As outlined above (and described in \Cref{2Subsec:cobordism}), a concordance of links in thickened surfaces \( \left( S , M \right) \) is descent if \( M \) does not contain an exceeding index \( 2 \) critical point; descent is the analogue of ribbon in the classical case. We are therefore interested in invariants that obstruct descent concordance but not ascent concordance.

Our main tool is an augmented version of Khovanov homology, defined by Manturov and the author \cite{Rushworth2017b}. It associates to a link in a thickened surface \( L \hookrightarrow \Sigma_g \times I \) and \( \gamma \in H^1 \left( \Sigma_g ; \Z_2 \right) \) a trigraded Abelian group, \( \dkh '' ( L, \gamma ) \), the \emph{totally reduced homology of \(L\) with respect to \( \gamma \)}. As described in \Cref{4Sec:ascent}, the totally reduced homology is invariant under certain concordances but not generic concordances. Crucially, it contains information regarding the intersection of \( L \) with attaching spheres of destabilizing handles (index \( 1 \) critical points of \( M \)). In \Cref{4Prop:gen0tnt} we show that this information is also robust under certain genus \( 0 \) cobordisms i.e.\ cobordisms of the form \( \left( S , M \right) \) with \( g ( S ) = 0 \).

Given concordant links \( L_1 \hookrightarrow \Sigma_{g_1} \times I \) and  \( L_2 \hookrightarrow \Sigma_{g_2} \times I \) with \( g_1 > g_2 \), the above properties allow us to prove that, if \(  \dkh '' ( L_1 , \gamma ) \) satisfies a certain condition, any concordance from \( L_1 \) to \( L_2 \) is ascent. Specifically, we show that \( L_1 \) cannot be made disjoint to the attaching sphere of a destabilizing handle within a descent concordance. It follows that if \( \left( S , M \right) \) is a concordance from \( L_1 \) to \( L_2 \) then \( M \) contains an index \( 2 \) critical point, exceeding by construction. If \( g_1 = g_2 \) and we can show that every concordance from \( L_1 \) to \( L_2 \) contains a (de)stabilizing handle (this can be done using \( \dkh '' ( L_1 , \gamma ) \) or \( \pi_1 ( \Sigma_{g_1} ) \), for example), then an identical argument shows that \( L_1 \) and \( L_2 \) are ascent concordant.

The totally reduced homology may be viewed as a generalization of the Rasmussen invariant extracted from the Lee homology of a knot in \( S^3 \) \cite{Rasmussen2010,Lee2005}. While the Rasmussen invariant contains information regarding the genus of surfaces appearing as cobordisms between two knots in \( S^3 \), \( \dkh '' ( L, \gamma ) \) contains information regarding the \(3\)-manifolds appearing in cobordisms between links in thickened surfaces (in addition to the cobordism surfaces).

Although our totally reduced homology is strong enough to detect ascent concordant links, there are a number of questions that may require even stronger invariants to be resolved.

\begin{question}\label{1Q:ascentslice}
Do there exist ascent concordant knots?
\end{question}
In particular, it is unknown if there exist knots that are ascent concordant to the unknot in \( \Sigma_0 \times I \). The totally reduced homology takes as input an element \( \gamma \in H^1 \left( \Sigma_g ; \Z_2 \right) \), and due to a result on the dimension of the totally reduced homology of a knot, the choice of \( \Z_2 \) coefficients renders it unsuited to the knot case. An upgrade of the construction that takes as input an element of the integral cohomology of \( \Sigma_g \) has the potential to detect ascent concordant knots.

Let \( L_1 \) and \( L_2 \) be concordant links. Denote by \( Ex ( L_1, L_2 ) \) the minimum number of exceeding index \( 2 \) critical points in a concordance between \( L_1 \) and \( L_2 \). In this paper we provide the first examples of pairs of links with \( Ex ( L_1, L_2 ) > 0 \). It is natural to ask if \( Ex ( L_1, L_2 ) \) need be arbitrarily large.
\begin{question}\label{1Q:arbitrary}
Given a positive integer \( m \), does there exist a pair of concordant nonsplit links, \( L_1 \) and \( L_2 \), with \( Ex ( L_1, L_2 ) > m \)?
\end{question}

This paper is organised as follows. In \Cref{2Sec:cobordism} we describe cobordism and concordance of links in thickened surfaces, and define ascent and descent concordance. \Cref{3Sec:trh} contains an overview of the construction of the totally reduced homology, and gives some of its properties. We employ the totally reduced homology in \Cref{4Sec:ascent}, and establish that the set of ascent concordant links is nonempty, using \Cref{4Thm:ascent}. We work in the smooth category throughout.

\subsubsection*{Acknowledgements} We thank Hans Boden and Andrew Nicas for their encouragement, and many helpful conversations and comments. We thank Robin Gaudreau, Gabriel Islambouli, Patrick Orson, and the anonymous referees for perspicacious comments on earlier versions of this work.

\section{Cobordism of links in thickened surfaces}\label{2Sec:cobordism}
In this section we define links in thickened surfaces and their diagrams, before describing cobordism and concordance of such objects. We also introduce the notions of \emph{ascent} and \emph{descent} cobordism.

\subsection{Links in thickened surfaces}\label{2Subsec:lits}
We denote by \( \Sigma_g \) a closed orientable surface of genus \( g \), not necessarily connected. For concreteness we state the definition of the genus of a disconnected surface. For \( S \) a closed orientable surface, the genus of \( S \) is given by
\begin{equation*}
	g ( S ) = \dfrac{2 c ( S ) - \chi ( S )}{2}
\end{equation*}
for \( c ( S ) \) the number of connected components of \( S \).

A \emph{link in a thickened surface} (henceforth simply \emph{link}) is an embedding \( L : \bigsqcup S^1 \hookrightarrow \Sigma_g \times I \), considered up to isotopy. We abuse notation to denote by \( L \hookrightarrow \Sigma_g \times I \) a link in \( \Sigma_g \times I \). Links in \( S^3 \) appear as links in \( \Sigma_0 \times I \), and are referred to as classical links. We refer to the unique knot in \( \Sigma_0 \times I \) that bounds a disc as \emph{the unknot}.

Given a link \( L \hookrightarrow \Sigma_g \times I \) a regular projection to \( \Sigma_g \) yields a \(4\)-valent graph on \( \Sigma_g \) whose vertices may be decorated with the under-\ or overcrossing decoration of classical knot theory. Such a decorated graph is known as a \emph{diagram of \( L \)}; an example is given in \Cref{2Fig:21lift}.

Two diagrams on \( \Sigma_g \) represent the same link if and only if they are related by a finite sequence of Reidemeister moves (where such moves occur in disc neighbourhoods on \( \Sigma_g \)). Notice that as we are considering links up to isotopy only, given two diagrams \( D_1 \looparrowright \Sigma_g \), \( D_2 \looparrowright \Sigma_g \), we may compare them - that is, pose the question `does \( D_1 \) represent the same link as \( D_2 \)?' - only when their ambient spaces are identical (as opposed to being merely diffeomorphic).

While other authors have considered links in thickened surfaces up to self-diffeomorphism of the (thickened) surface, working as we do at the level of isotopy is well-established. See Asaeda-Przytycki-Sikora \cite{Asaeda2004}, Queffelec-Wedrich \cite{Wedrich18}, and references therein, for example.

\begin{figure}
	\includegraphics[scale=0.65]{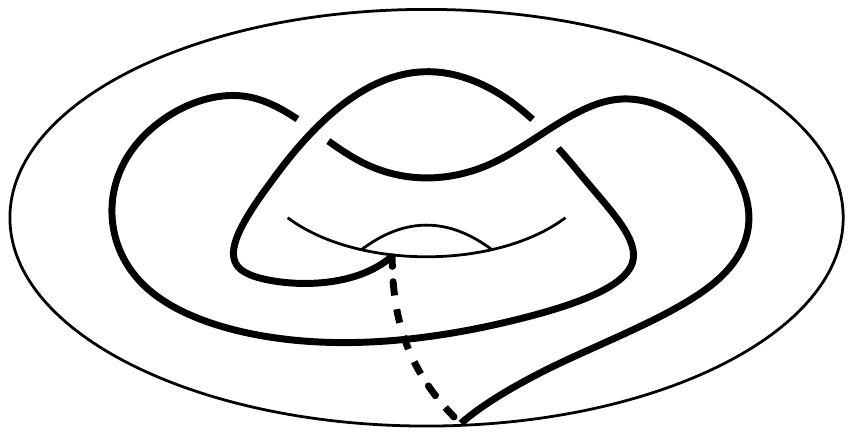}
	\caption{A diagram of a knot in \( \Sigma_1 \times I \).}
	\label{2Fig:21lift}
\end{figure}

\subsection{Cobordism}\label{2Subsec:cobordism}
We define cobordism and concordance of links in thickened surfaces, following Turaev \cite{TuraevCob} (note that he uses the term cobordism for what we refer to as a concordance).

\begin{definition}[Cobordism]\label{2Def:cob}
Let \( L_1 \hookrightarrow \Sigma_{g_1} \times I \) and \( L_2\hookrightarrow \Sigma_{g_2} \times I \) be links.
A \emph{cobordism} from \( L_1 \) to \( L_2 \) is a pair, \( (S, M) \), consisting of a compact orientable \(3\)-manifold \( M \), with \( \partial M = \Sigma_{g_1} \sqcup \Sigma_{g_2} \), and a compact orientable surface \( S \) properly embedded in \( M \times I \), with \( \partial S = L_1 \sqcup L_2 \). If such a pair \( (S,M) \) exists we say that \( L_1\) and \( L_2 \) are \emph{cobordant}. We refer to \( S \) as the \emph{cobordism surface}.

If \( S \) is a disjoint union of annuli such that each annulus has a boundary component in both \( \Sigma_{g_1} \times I \) and \( \Sigma_{g_2} \times I \), we say that \( (S,M) \) is a \emph{concordance}, and that \( L_1\) and \( L_2 \) are \emph{concordant}.
\end{definition}
A schematic picture of a cobordism is given in \Cref{1Fig:schematic}. Notice that if \(L_1\) and \(L_2\) are concordant then \( | L_1 | = | L_2 | \) (for \( | L | \) the number of components of \( L \)). There is an important distinction between concordance and genus \(0\) cobordism: two links and are genus \(0\) cobordant if there is a cobordism between them, \( (S,M) \), with \( g(S)=0 \). Notice that it is possible that \( | L_1 | \neq | L_2 | \) and that \( S \) is not a disjoint union of annuli in this case.

\begin{figure}
\includegraphics[scale=0.5]{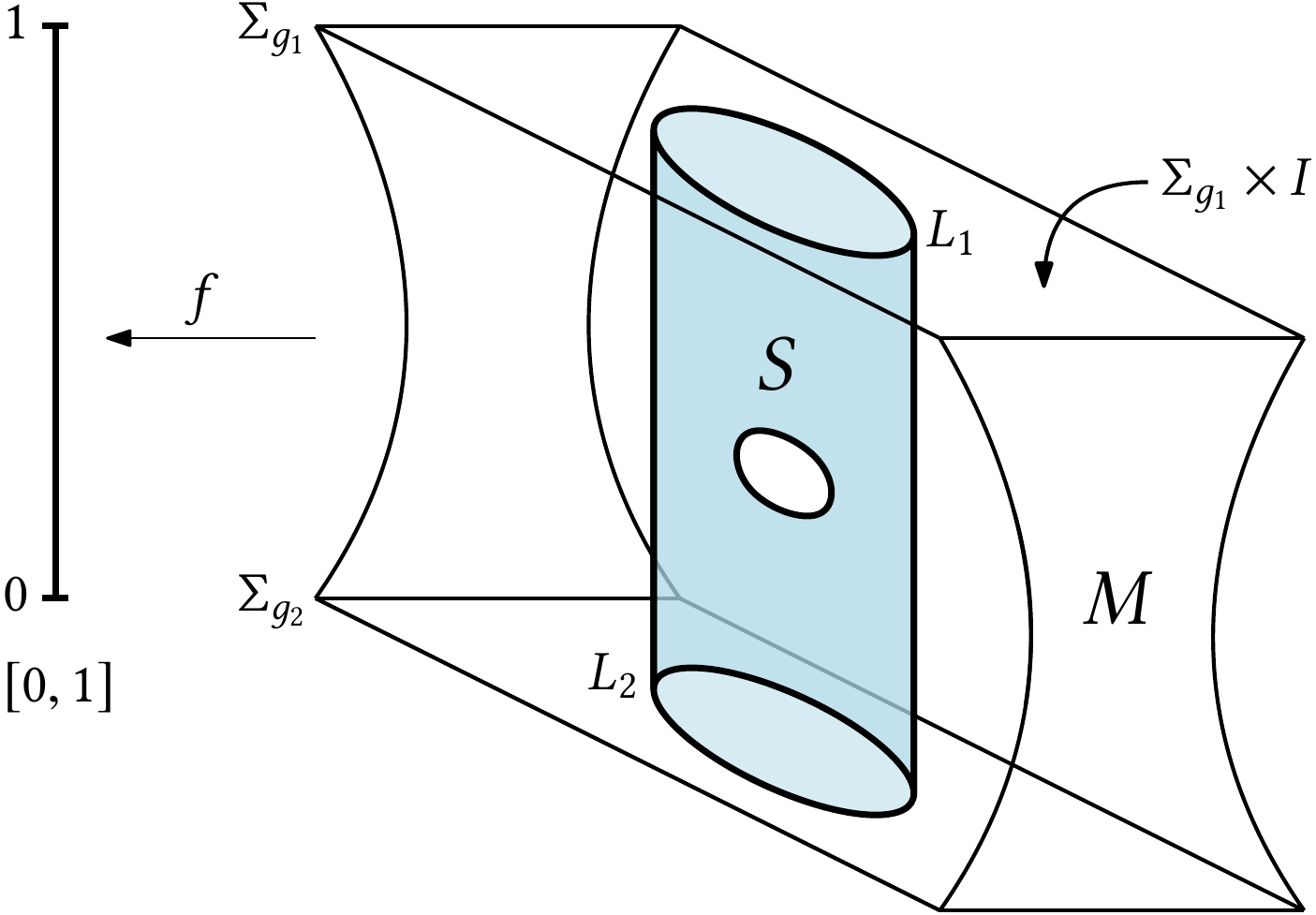}
\caption{A cobordism between links in thickened surfaces.}
\label{1Fig:schematic}
\end{figure} 

\begin{proposition}\label{2Prop:cobunknot}
Any two links are cobordant.
\end{proposition} 
\begin{proof}
We prove that any link is cobordant to the unknot. The proposition then follows by transitivity.

Given a diagram \( D \) of \( L \hookrightarrow \Sigma_g \times I \), one may remove all of its crossings via repeated iterations of a cobordism \( ( S , \Sigma_g \times I ) \), where \( S \) is constructed as follows:
\begin{center}
\includegraphics[scale=1]{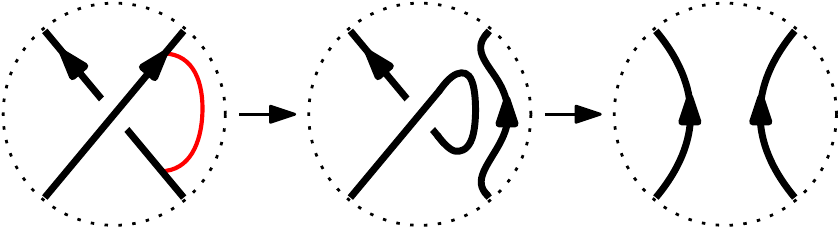}
\end{center}
(this will add genus to the cobordism surface, in general). The resulting diagram, \( D ' \), is a disjoint union of circles on \( \Sigma_g \). Thus there is an available sequence of destabilizations taking \( \Sigma_g \) to \( \Sigma_0 \), the attaching spheres of which are disjoint to \( D ' \). After making these destabilizations we are left with a disjoint union of circles on \( \Sigma_0 \); cap off all but one of these circles. We have just described a cobordism from \( L \) to the unique knot in \( \Sigma_0 \times I \) that bounds a disc, the unknot.
\end{proof}
For an alternative proof of \Cref{2Prop:cobunknot} see \cite{Kauffman1998c}.

In contrast, not all links are concordant. For instance, it can be shown that the knot depicted in \Cref{2Fig:21lift} is not concordant to any knot in \( \Sigma_0 \times I \) \cite[Section \(4.3\)]{Rushworth2017}.

As described in \Cref{1Sec:intro}, the goal of this paper is to show that there exist pairs of concordant links such that any concordance between them, \( ( S, M ) \), passes through surfaces of genus higher than that of the initial and terminal surface. This is equivalent to \( M \) possessing a certain type of index \(2\) Morse critical point.

We now concretise the notion of `passing through' used above.
\begin{definition}[Exceeding critical point]\label{2Def:exceeding}
Let \( (S,M) \) be a cobordism from \( L_1 \hookrightarrow \Sigma_{g_1} \times I \) to \( L_2 \hookrightarrow \Sigma_{g_2} \times I \). Let \( f \) be a Morse function on \( M \). Up to isotopy we may assume that \( S \) is transverse to the \( I \) factor of \( M \times I \). That is, if \( t \) is a regular value of \( f \), then \( S \) intersects \( f^{-1} \left( t \right) \times I \) transversely.

Under the convention \( f \left( \Sigma_{g_1} \right) = \lbrace 1 \rbrace \) and \( f \left( \Sigma_{g_2} \right) = \lbrace 0 \rbrace \), and reading the cobordism by starting at \(L_1 \) and ending at \( L_2 \), an index \(2\) critical point of \( f \) is a \(3\)-dimensional \(1\)-handle addition. From this viewpoint the genus of surfaces appearing as level sets increases by \(1\) when passing an index \(2\) critical point. Similarly, the genus of surfaces appearing as level sets decreases by \(1\) when passing an index \(1\) critical point. See \Cref{1Fig:levelsets}.

Suppose \( p \in M \) is an index \(2\) critical point of \(f\). We say that \( p \) is \emph{exceeding} if \( f^{-1} \left( f(p)-\epsilon \right) = \Sigma_{k} \) with \( k > g_1, g_2 \) for \( \epsilon > 0 \) arbitrarily small. That is, \(p\) is exceeding if the genus of level surfaces appearing immediately after \(p\) is greater than both \( g_1 \) and \(g_2\).
\end{definition}

The convention that \( f \left( \Sigma_{g_1} \right) = \lbrace 1 \rbrace \) and \( f \left( \Sigma_{g_2} \right) = \lbrace 0 \rbrace \) follows that given by Gordon in the case of ribbon concordance of classical knots \cite{Gordon1981} (further employed in \cite{Baker16,Hayden19,LivingstonNaik,Miyazaki98,Miyazaki18} among others\footnote{However, in some recent papers another convention has been used \cite{Levine19,Zemke19}.}). Under this convention index \( i \) critical points of \( M \) correspond to \(3\)-dimensional \(( 3 - i )\)-handle additions. As described above, an index \( 1 \) critical point corresponds to a \(3\)-dimensional \(2\)-handle addition, with attaching sphere a simple closed curve on the level surface immediately preceding the critical point. Similarly an index \( 2 \) critical point corresponds to a \(3\)-dimensional \(1\)-handle addition. We say that a \(2\)-handle addition is \emph{destabilizing} if it reduces the genus of the level surface, and that a \(1\)-handle addition is \emph{stabilizing} if it increases the genus of the level surface. Henceforth we shall not distinguish between critical points and the handle additions they correspond to.

\begin{figure}
\includegraphics[scale=0.5]{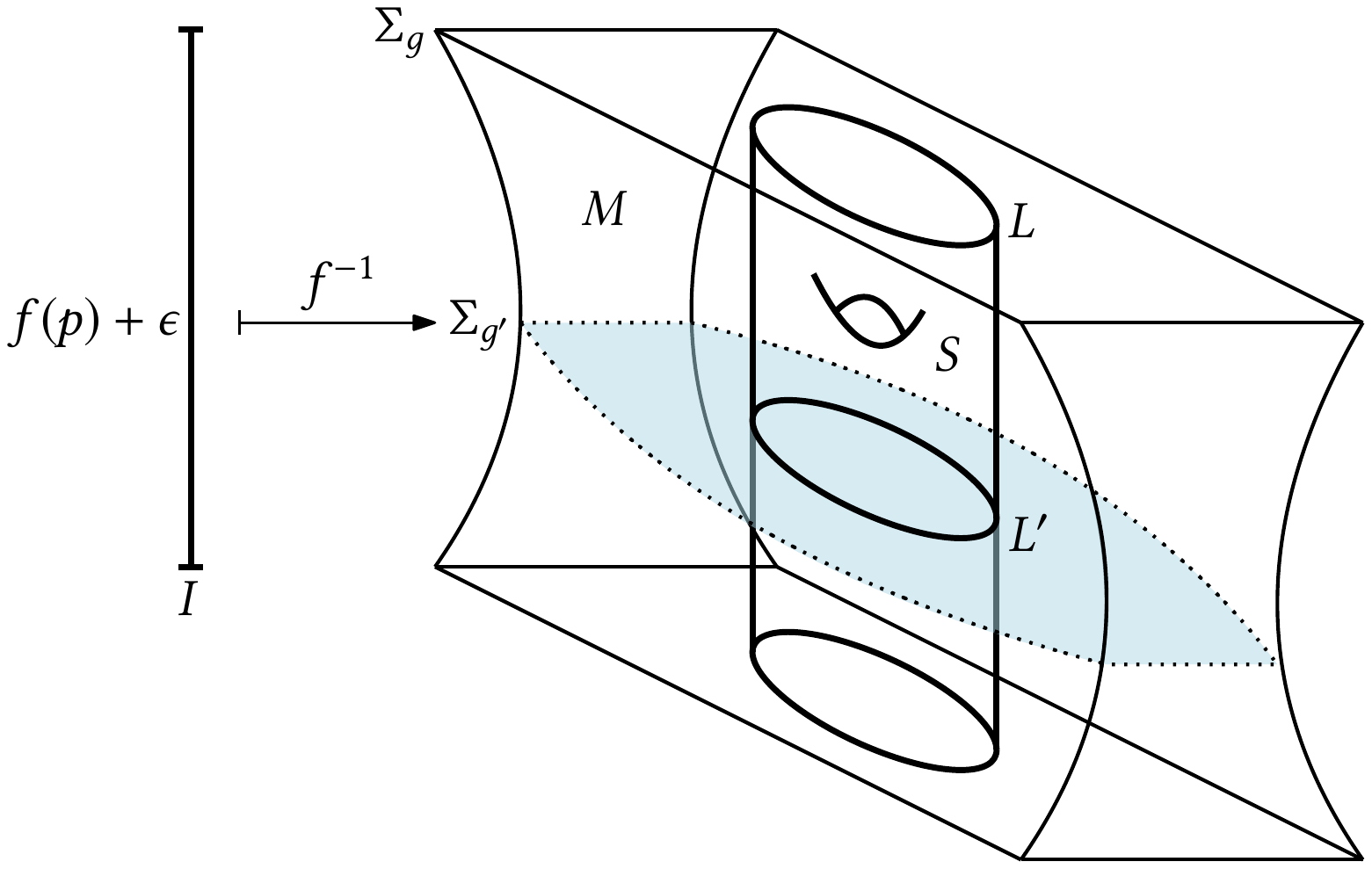}
\caption{A thickening of a level set of \(f \), in blue.}
\label{1Fig:levelsets}
\end{figure}

\begin{definition}[Pseudostrict, strict cobordism]\label{2Def:strict}
Let \( (S,M) \) be a cobordism. We say that \( (S,M) \) is \emph{strict} if \( M = \Sigma_g \times I \). We say that \( ( S, M ) \) is \emph{pseudostrict} if \( M \) admits a Morse function with critical points of the following types: index \( 0 \), index \( 1 \) with attaching sphere a separating curve, index \( 2 \) such that the handle is attached between disjoint components, or index \( 3 \).
\end{definition}
Notice that the genus of level surfaces does not change within a pseudostrict cobordism, but the number of connected components may.

\begin{definition}[Ascent, descent cobordism]
We say that \( (S,M) \) is \emph{descent} if \(M\) admits a Morse function without exceeding index \(2\) critical points. We say that \( (S,M) \) is \emph{ascent} if \(M\) admits a Morse function with an exceeding index \(2\) critical point.
\end{definition}

Notice that a pseudostrict cobordism is descent, but that the converse is not necessarily true. Two links are said to be \emph{descent\slash pseudostrictly\slash strictly} cobordant if there exists a descent\slash pseudostrict\slash strict cobordism between them. Two links are said to be \emph{ascent cobordant} if they are cobordant but not descent cobordant. Ascent\slash descent\slash pseudostrict\slash strict concordance and genus \( 0 \) cobordism are defined likewise, so that two links are said to be \emph{ascent concordant} if they are concordant but not descent concordant.

\section{Totally reduced homology}\label{3Sec:trh}
We outline the difficulties encountered when extending Khovanov homology to links in thickened surfaces in \Cref{3Subsec:extending}, before reviewing the construction of the totally reduced homology in \Cref{3Subsec:doubled,3Subsec:perturb}. After describing the functorial nature of the theory in \Cref{3Subsec:func}, we highlight some of its important properties in \Cref{3Subsec:properties}.

\subsection{Extending Khovanov homology to thickened surfaces}\label{3Subsec:extending}
When constructing Khovanov homology for classical links the cube of resolutions possesses exactly two types of edge:
\begin{enumerate}[(i)]
\item An edge along which one circle splits into two.
\item An edge along which two circles merge into one.
\end{enumerate}
Passing to links in thickened surfaces causes a new type of edge to appear, along which one circle can be sent to one circle, as depicted in \Cref{3Fig:121} (the associated cobordism is a once punctured M\"{o}bius band). This is known as a \emph{single cycle smoothing}. A map must be assigned to these edges, that we denote \( \eta \). Let \( \mathcal{A} \) be the module assigned to one circle in the construction of classical Khovanov homology; if one attempts to assign \( \mathcal{A} \) to circles in this new situation, the map \( \eta : \mathcal{A} \rightarrow \mathcal{A} \) is forced to be the zero map for quantum grading reasons. This causes collateral damage to the chain complex, however, so that it is no longer well-defined (over any coefficient ring except \( \Z_2 \)). Extra technology must then be added to repair this, as is done by Manturov \cite{Manturov2006} and Tubbenhauer \cite{Tubbenhauer2014a} (see also \cite{Dye2014,Asaeda2004}).

We take a different approach, altering the module assigned to a circle. Specifically, we assign \( \mathcal{A} \oplus \left( \mathcal{A} \lbrace -1 \rbrace \right) \) (where \( \lbrace -1 \rbrace \) denotes a quantum grading shift by \( -1 \)). As is detailed in \cite{Rushworth2017}, this allows \( \eta \) to be nonzero, and yields a well-defined homology theory automatically; this homology theory is known as \emph{doubled Khovanov homology}. In the remainder of this section we describe an augmentation of doubled Khovanov homology that is more sensitive to the ambient thickened surface.

\begin{figure}
	\centering
	\begin{tikzpicture}[scale=1,
	roundnode/.style={}]
	
	\node[roundnode] (s0)at (-3,0)  {
		\includegraphics[scale=0.5]{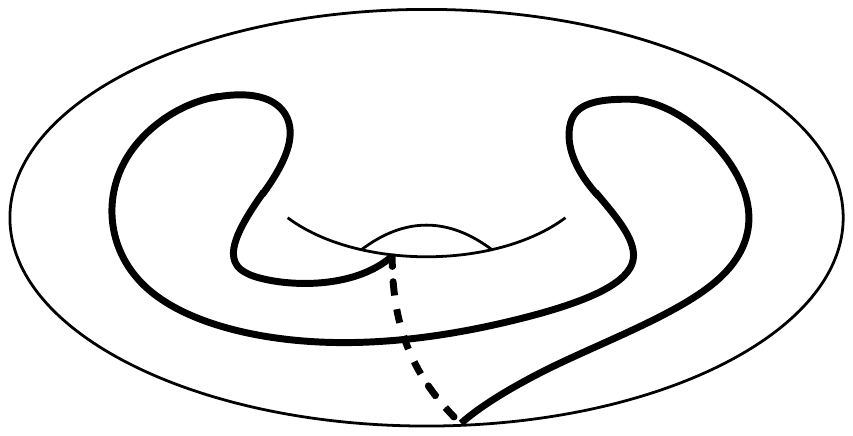}};
	
	\node[roundnode] (s1)at (3,0)  {
		\includegraphics[scale=0.5]{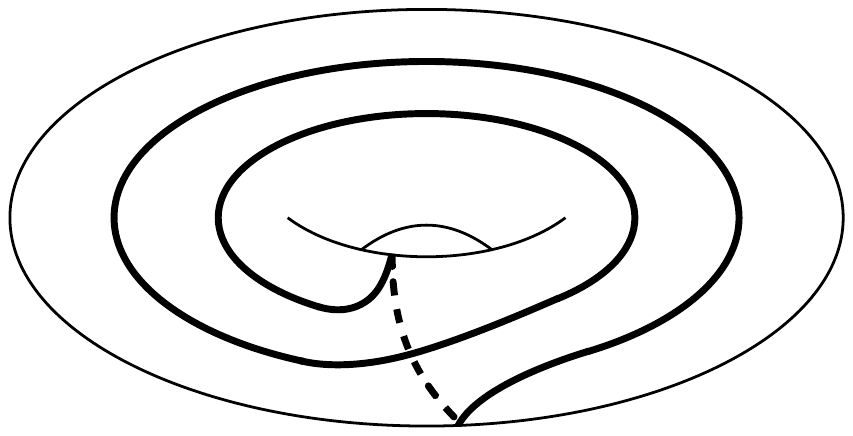}};
	
	\draw[->,thick] (s0)--(s1) node[above,pos=0.5]{\( \eta \)} ;
	\end{tikzpicture}
	\caption{The single cycle smoothing.}\label{3Fig:121}
\end{figure}

\begin{remark}
Doubled Khovanov homology has structural similarities to an instanton homology due to Kronheimer and Mrowka \cite{Kronheimer2019,Kronheimer2019b}, in that both contain information regarding the Tait colourings of trivalent graphs. Indeed, the four colour theorem may be restated in terms of the ranks of (perturbations of) either of these homology theories.
\end{remark}

\subsection{Doubled Khovanov homology of links in thickened surfaces}\label{3Subsec:doubled}
The definition of the totally reduced homology is given in \cite[Section \(3\)]{Rushworth2017b}, which itself relies on a generalization of doubled Khovanov homology \cite{Rushworth2017}. The construction is familiar from other theories in the Khovanov tradition: a cube of smoothings is associated to a diagram, and then turned into an algebraic chain complex. The chain homotopy equivalence class of this chain complex is an invariant of the link represented by the diagram, so that its homology is also.

\begin{definition}[Smoothing]
	\label{3Def:smoothing}
	Let \( D \) be a diagram of an oriented link \( L \hookrightarrow \Sigma_g \times I \). The crossings of \( D \) may be \emph{resolved} in one of two ways:
	\begin{center}
	\includegraphics[scale=0.75]{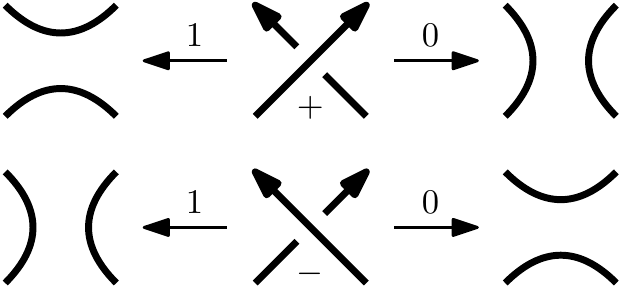}
	\end{center}
	The resolutions are known as the \(0\)-\ and the \(1\)-resolution, as depicted. A \emph{smoothing} of \( D \) is a diagram formed by arbitrarily resolving all of the crossings of \( D \); it is a disjoint union of circles in \( \Sigma_g \).

	Given a smoothing \( \mathscr{S} \) of \( D \), the \emph{height of \( \mathscr{S} \)}, denoted \( | \mathscr{S} | \), is defined as
	\begin{equation*}
	| \mathscr{S} | \coloneqq \# \left( 1 \text{-resolutions in}~ \mathscr{S} \right) - n_-
	\end{equation*}
	for \( n_- \) the number of negative crossings of \( D \).
\end{definition} 

First, smoothings of diagrams are used to decorate the vertices of an appropriate-dimensional cube.

\begin{definition}[Dotted cube of smoothings]
	\label{3Def:dottedcube}
	Let \( D \) be a diagram of an oriented link \( L \hookrightarrow \Sigma_g \times I \), with \( n \) crossings, of which \( n_- \) are negative. Arbitrarily label the crossings from \(1 \) to \( n \). Denote by \( e_1 e_2 \cdots e_n \in \lbrace 0, 1 \rbrace^{\times n} \) the smoothing obtained by resolving the \( k \)-th crossing into its \( e_k \)-resolution. Assign to the vertices of the cube \( \lbrace 0, 1 \rbrace^{\times n} \) the appropriate smoothings of \( D \); we no longer make a distinction between a vertex and the smoothing assigned to it. The result is known as the \emph{cube of smoothings of \( D \)}.

	Pick \( \gamma \in H^1 \left( \Sigma_g ; \Z_2 \right) \). Given a smoothing \( \mathscr{S} \) of \( D \), a circle within \( \mathscr{S} \) is decorated with a \emph{dot} if it has nonzero image under \( \gamma \). Repeat this for all of the vertices of the cube. The resulting assignment of dots is known as \emph{the dotting with respect to \( \gamma \)}. The fully decorated cube is referred to as the \emph{dotted cube of smoothings of \( D \) with respect to \( \gamma \)}, and is denoted \( \llbracket D, \gamma \rrbracket \).
\end{definition}
Two examples of dottings are given in \Cref{3Fig:d21cube}; green dots represent the dotting associated to Poincar\'e dual to the green simple closed curve, and the red simple closed curve does not produce any dots. 

\begin{figure}
	\begin{tikzpicture}[scale=1,
	roundnode/.style={}]
	
	\node[roundnode] (s0)at (-6.5,0)  {\includegraphics[scale=0.5]{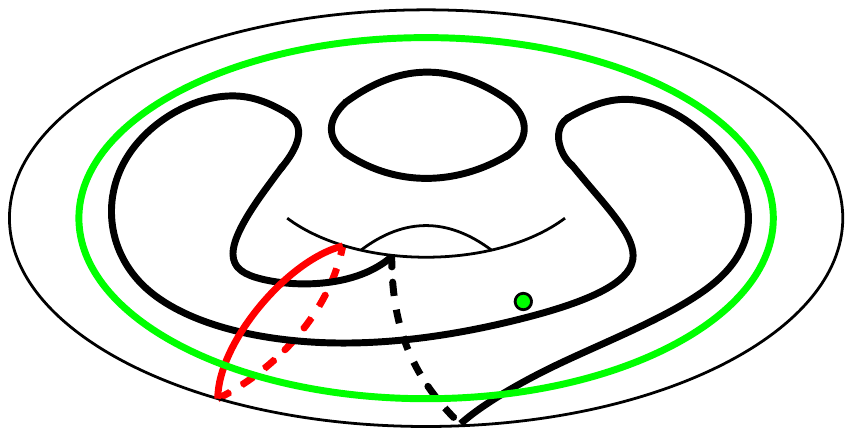}
	};
	
	\node[roundnode] (s1)at (-2,2.75)  {\includegraphics[scale=0.5]{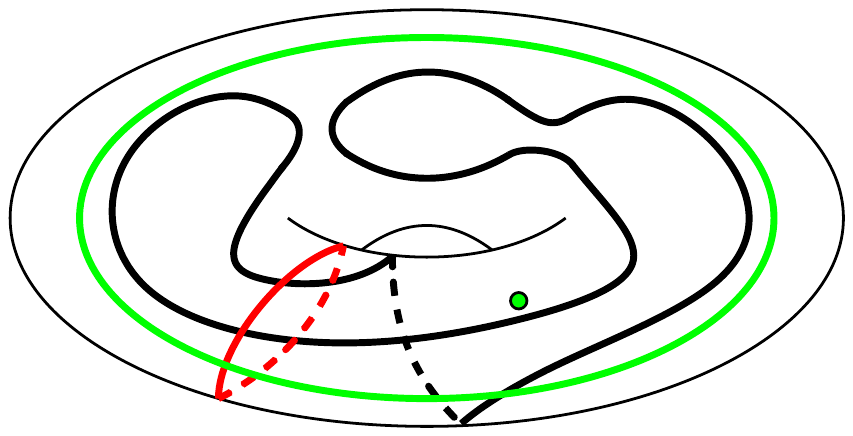}
	};
	
	\node[roundnode] (s2)at (-2,-2.75)  {\includegraphics[scale=0.5]{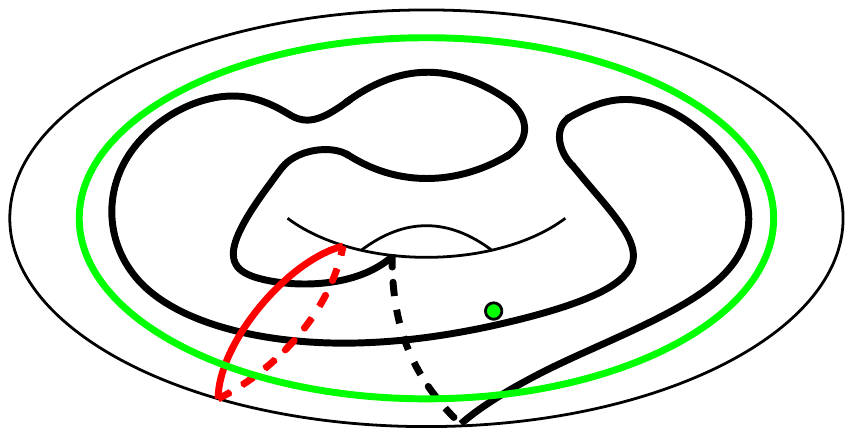}
	};
	
	\node[roundnode] (s3)at (2.5,0)  {\includegraphics[scale=0.5]{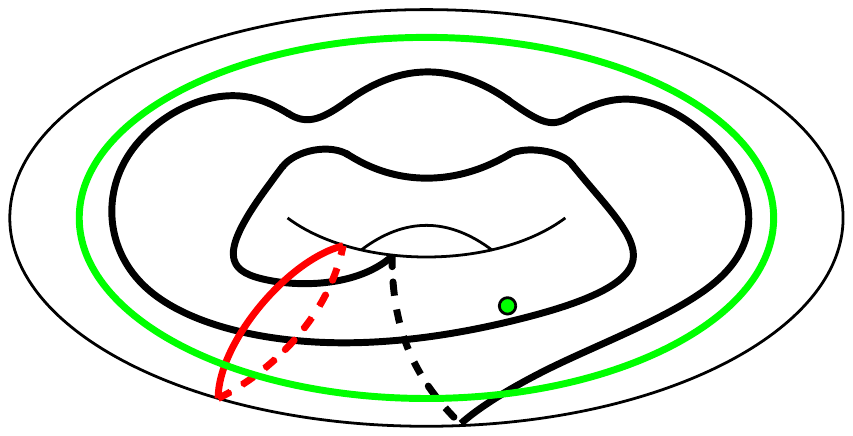}
	};

	\draw[->,thick] (s0)--(s1) node[above left,pos=0.6]{} ;
	
	\draw[->,thick] (s0)--(s2) node[below left,pos=0.6]{} ;
	
	\draw[->,thick] (s1)--(s3) node[above right,pos=0.4]{} ;
	
	\draw[->,thick] (s2)--(s3) node[below right,pos=0.3]{} ;			
	\end{tikzpicture}
	\caption{The dotted cube of smoothings of the diagram depicted in \Cref{2Fig:21lift}.}
	\label{3Fig:d21cube}
\end{figure}

Next, the fully decorated cube is converted into a chain complex.

\begin{definition}[Dotted complex]\label{3Def:dottedcomplex}
Let \( D \) be a diagram of an oriented link \( L \hookrightarrow \Sigma_g \times I \). Pick \( \gamma \in H^1 ( \Sigma_g ; \Z_2) \) and form the dotted cube \( \llbracket D, \gamma \rrbracket \) as in \Cref{3Def:dottedcube}. 

Let \( \mathscr{S} \) be a vertex of \( \llbracket D, \gamma \rrbracket \), made up of \( k \) circles. We assign a vector space to \(\mathscr{S} \) in the following manner
\begin{equation*}
\bigsqcup_k \bigcirc^{\left( \bullet \right)} \longmapsto \overset{\left( \bullet \right)\phantom{\left( \bullet \right)}}{\mathcal{A}^{\otimes k}} \oplus \left( \left( \overset{\left( \bullet \right)\phantom{\left( \bullet \right)}}{\mathcal{A}^{\otimes k}} \right) \left\lbrace -1 \right\rbrace \right)
\end{equation*}
where
\begin{equation*}
\mathcal{A} = {\Q [X]} / X^2 = \left\langle v_+ , v_- \right\rangle_{\Q}.\footnote{The construction is valid for coefficients in a commutative unital ring, but we shall only need \( \Q \).}
\end{equation*}
The vector space \( \mathcal{A} \) is graded with \( v_{\pm} \) of degree \( \pm 1 \); this grading extends linearly across tensor products. The braces \( \left\lbrace -1 \right\rbrace \) denote a grading shift by \( - 1 \).

Arbitrarily identify the tensorands and the circles of the smoothing: a dot, \( \bullet \), is added to the vector space if the associated circle possesses a dot. This decoration persists to the elements so that
\begin{equation*}
\overset{\bullet}{\mathcal{A}} = \left\langle v_{\dpl} , v_{\dm} \right\rangle_{\Q}.
\end{equation*}

We add a superscript to denote the shifted and unshifted copies of \( \mathcal{A} \). Specifically, we write
\begin{equation*}
\overset{\left( \bullet \right)}{\mathcal{A}} = \left\langle v^{\text{u}}_{\overset{\left( \bullet \right)}{+}} , v^{\text{u}}_{\overset{\left( \bullet \right)}{-}} \right\rangle_{\Q}
\end{equation*}
and
\begin{equation*}
\overset{\left( \bullet \right)}{\mathcal{A}} \left\lbrace -1 \right\rbrace = \left\langle v^{\ell}_{\overset{\left( \bullet \right)}{+}} , v^{\ell}_{\overset{\left( \bullet \right)}{-}} \right\rangle_{\Q}
\end{equation*}
and similarly for tensor products. A dot in parentheses, \( \left( \bullet \right) \), denotes a copy of \( \mathcal{A} \) that may or may not be dotted (likewise for elements such as \( v^{\ell}_{\overset{\left( \bullet \right)}{-}} \)).

Denote by \( \cdkh_i \left( D, \gamma \right) \) the direct sum of the vector spaces assigned to the vertices of height \( i \). The \emph{doubled Khovanov complex of \( D \) with respect to \( \gamma \)}, denoted \( \cdkh \left( D, \gamma \right) \) has chain spaces \( \cdkh_i \left( D, \gamma \right) \), and differentials matrices of maps, whose entries are determined by the edges of \( \llbracket D, \gamma \rrbracket \). The forms of these maps depend on the dotting, and are given by the maps \( m^0 \), \( \Delta^0 \), and \( \eta^0 \) in \Cref{3Fig:diffcomp}. Signs are added to the entries in the standard way.
\end{definition}

The assignment used in \Cref{3Def:dottedcomplex} is not a topological quantum field theory (TQFT) nor an unoriented TQFT in the sense of Turaev and Turner \cite{Turaev2006}: it fails the multiplicativity axiom. Nevertheless, it is functorial with respect to link cobordism, which we exploit in \Cref{4Sec:ascent}.

\begin{figure}
	\begin{tikzpicture}[scale=0.6,
	roundnode/.style={}]
	\node[roundnode] (s5)at (-8.5,-7) {\( \begin{matrix}
		\mathcal{A} \otimes \overset{{\color{green} \bullet}}{\mathcal{A}} \\
		\oplus \\
		\left (\mathcal{A} \otimes \overset{{\color{green} \bullet}}{\mathcal{A}} \right) \lbrace -1 \rbrace \\
		\end{matrix}\)
	};
	
	\node[roundnode] (s6)at (-1.5,-7) {\( \begin{matrix}
		\left(\begin{matrix}\overset{{\color{green} \bullet}}{\mathcal{A}} \\
		\oplus \\
		\overset{{\color{green} \bullet}}{\mathcal{A}} \lbrace -1 \rbrace\end{matrix}\right) \\
		\oplus \\
		\left(\begin{matrix}\overset{{\color{green} \bullet}}{\mathcal{A}} \\
		\oplus \\
		\overset{{\color{green} \bullet}}{\mathcal{A}} \lbrace -1 \rbrace\end{matrix}\right)
		\end{matrix} \)
	};
	
	\node[roundnode] (s7)at (5.5,-7) {\( \begin{matrix}
		\overset{{\color{green} \bullet}}{\mathcal{A}} \\
		\oplus \\
		\overset{{\color{green} \bullet}}{\mathcal{A}} \lbrace -1 \rbrace \\
		\end{matrix}\)
	};
	
	\node[roundnode] (s8)at(-8.5,-10.5) {\( -2 \)};
	
	\node[roundnode] (s9)at(-1,-10.5) {\( -1 \)};
	
	\node[roundnode] (s10)at(5.5,-10.5) {\( 0 \)};
	
	\draw[->,thick] (s5)--(s6) node[above,pos=0.5]{\( d_{-2} =  \begin{pmatrix}
		m^{0} \\
		m^{0}
		\end{pmatrix} \)
	};
	
	\draw[->,thick] (s6)--(s7) node[above,pos=0.5]{\( d_{-1} =  \left( \eta^{0}, - \eta^{0} \right) \)
	};
	\end{tikzpicture}
	\caption{The dotted complex associated to the cube of smoothings depicted in \Cref{3Fig:d21cube}, with respect to the Poincar\'e dual to the green simple closed curve.}
	\label{3Fig:dottedcomp}
\end{figure}

There is a distinguished basis of \( \cdkh \left( D, \gamma \right) \), on which we define three gradings.

\begin{definition}[States]\label{3Def:states}
Let \( D \) be a diagram of an oriented link \( L \hookrightarrow \Sigma_g \times I \). Let \( \mathscr{S} \) be a smoothing of \( D \) whose circles are decorated with exactly one of \( + \) and \( - \) (in addition to the dotting with respect to \( \gamma \)). A \emph{state} is an element \( a^{\text{u/l}} \in \cdkh \left( D, \gamma \right) \) of the form
\begin{equation*}
	\begin{aligned}
		a^{\text{u}} &= v^{\text{u}}_{\overset{\left(\bullet\right)}{\pm}} \otimes v^{\text{u}}_{\overset{\left(\bullet\right)}{\pm}} \otimes \cdots \otimes v^{\text{u}}_{\overset{\left(\bullet\right)}{\pm}} \\
		a^{\ell} &= v^{\ell}_{\overset{\left(\bullet\right)}{\pm}} \otimes v^{\ell}_{\overset{\left(\bullet\right)}{\pm}} \otimes \cdots \otimes v^{\ell}_{\overset{\left(\bullet\right)}{\pm}}
	\end{aligned}
\end{equation*}
where the \( \pm \) and \( \left( \bullet \right) \) are determined by the decorations of the associated circle of \( \mathscr{S} \) (under the identification of the circles of \( \mathscr{S} \) and the tensorands as described in \Cref{3Def:dottedcomplex}).
\end{definition}

\begin{definition}[Gradings]\label{3Def:gradings}
Let \( D \) be a diagram of an oriented link \( L \hookrightarrow \Sigma_g \times I \) with \( n_- \) negative crossings and writhe \( wr(D) \). Pick \( \gamma \in H^1 ( \Sigma_g ; \Z_2) \) and form the dotted complex \( \cdkh \left( D, \gamma \right) \) as in \Cref{3Def:dottedcomplex}.

We define three gradings on the states of \( \cdkh \left( D, \gamma \right) \). Let \( a^{\text{u/l}} \) be a state associated to the smoothing \( \mathscr{S} \). The \emph{homological grading}, \( i \), is defined as
\begin{equation}\label{3Eq:homgrading}
	i ( a^{\text{u} / \ell} ) = | \mathscr{S} | = \# \left( 1\text{-resolutions in}~ \mathscr{S} \right) - n_- .
\end{equation}
The \(i\)-grading does not depend on the \( \text{u} / \ell \) superscript, nor the dotting associated to \( \gamma \). The \emph{quantum grading}, \( j \), is defined as
\begin{equation}\label{3Eq:quantumgrading}
	\begin{aligned}
	j ( a^{\text{u}} ) &= \# \left( v^{\text{u}}_{\overset{\left(\bullet\right)}{+}} \text{'s in}~a^{\text{u}} \right) - \# \left( v^{\text{u}}_{\overset{\left(\bullet\right)}{-}} \text{'s in}~a^{\text{u}} \right) + i \left( a^{\text{u} } \right) + wr(D) \\
	j ( a^{\ell} ) &= \# \left( v^{\ell}_{\overset{\left(\bullet\right)}{+}} \text{'s in}~a^{\ell} \right) - \# \left( v^{\ell}_{\overset{\left(\bullet\right)}{-}} \text{'s in}~a^{\ell} \right) + i ( a^{\ell } ) + wr(D) - 1
	\end{aligned}
\end{equation}
(note this is simply the grading of \( \mathcal{A}^{\otimes k} \) described in \Cref{3Def:dottedcomplex} with a particular shift). The \( j \)-grading depends on the \( \text{u} / \ell \) superscript, but not the dotting associated to \( \gamma \). The \emph{dotted grading}, \( c \), is defined as
\begin{equation}\label{3Eq:dottedgrading}
	c ( a^{\text{u} / \ell} ) = \# \left( v^{\text{u} / \ell}_{\dm} \text{'s in}~a^{\text{u} / \ell} \right) - \# \left( v^{\text{u} / \ell}_{\dpl} \text{'s in}~a^{\text{u} / \ell} \right) + \dfrac{1}{2} j ( a^{\text{u} / \ell} ) .
\end{equation}
The \( c \)-grading depends on both the \( \text{u} / \ell \) superscript and the dotting associated to \( \gamma \). The \( i \)-\ and \( j \)-gradings are \( \Z \)-gradings, while the \( c \)-grading is a \( \Z \left[ \frac{1}{2} \right] \)-grading.
\end{definition}

\begin{figure}
\begin{minipage}{\textwidth}\small
We denote by \( \Delta : \underline{\phantom{X}} \rightarrow \bullet \otimes \bullet \) a \( \Delta \) map taking an undotted circle to two dotted circles, and so forth. The \( \text{u} / \ell \) superscripts are suppressed for the \( m \) and \( \Delta \) maps, as they do not interact with them.

\begin{equation*}\label{3Eq:dkkdiff1}
	m : \underline{\phantom{X}} \otimes \underline{\phantom{X}} \rightarrow \underline{\phantom{X}} \left\lbrace
	\begin{aligned}
		v_{++} & \xrightarrow{m^{0}} v_+ \qquad & v_{++} & \xrightarrow{m^{+2}} 0 \qquad & v_{++} & \xrightarrow{m^{+4}} 0 \\
		v_{+-}, v_{-+} & \xrightarrow{m^{0}} v_- \qquad & v_{+-}, v_{-+} & \xrightarrow{m^{+2}} 0 \qquad & v_{+-}, v_{-+} & \xrightarrow{m^{+4}} 0 \\
		v_{--} & \xrightarrow{m^{0}} 0 \qquad & v_{--} & \xrightarrow{m^{+2}} 0 \qquad & v_{--} & \xrightarrow{m^{+4}} v_+
	\end{aligned}\right.	
\end{equation*}

\begin{equation*}\label{3Eq:dkkdiff2}
	m : \bullet \otimes \bullet \rightarrow \underline{\phantom{X}} \left\lbrace
	\begin{aligned}
		v_{\dpl \dpl} & \xrightarrow{m^{0}} 0 \qquad & v_{\dpl \dpl} & \xrightarrow{m^{+2}} v_+ \qquad & v_{\dpl \dpl} & \xrightarrow{m^{+4}} 0 \\
		v_{\dpl \dm}, v_{\dm \dpl} & \xrightarrow{m^{0}} v_- \qquad & v_{\dpl \dm}, v_{\dm \dpl} & \xrightarrow{m^{+2}} 0 \qquad & v_{\dpl \dm}, v_{\dm \dpl} & \xrightarrow{m^{+4}} 0 \\
		v_{\dm \dm} & \xrightarrow{m^{0}} 0 \qquad & v_{\dm \dm} & \xrightarrow{m^{+2}} 0 \qquad & v_{\dm \dm} & \xrightarrow{m^{+4}} v_+
	\end{aligned}\right.	
\end{equation*}

\begin{equation*}\label{3Eq:dkkdiff3}
	m : \bullet \otimes \underline{\phantom{X}} \rightarrow \bullet \left\lbrace
	\begin{aligned}
		v_{\dpl +} & \xrightarrow{m^{0}} v_{\dpl} \qquad & v_{\dpl +} & \xrightarrow{m^{+2}} 0 \qquad & v_{\dpl +} & \xrightarrow{m^{+4}} 0 \\
		v_{\dpl -} & \xrightarrow{m^{0}} 0 \qquad & v_{\dpl -} & \xrightarrow{m^{+2}} v_{\dm} \qquad & v_{\dpl -} & \xrightarrow{m^{+4}} 0 \\
		v_{\dm +} & \xrightarrow{m^{0}} v_{\dm} \qquad & v_{\dm +} & \xrightarrow{m^{+2}} 0 \qquad & v_{\dm +} & \xrightarrow{m^{+4}} 0 \\
		v_{\dm -} & \xrightarrow{m^{0}} 0 \qquad & v_{\dm -} & \xrightarrow{m^{+2}} 0 \qquad & v_{\dm -} & \xrightarrow{m^{4}} v_{\dpl}
	\end{aligned}\right.	
\end{equation*}
\noindent\hrulefill
\begin{equation*}\label{3Eq:dkkdiff4}
	\Delta : \underline{\phantom{X}} \rightarrow \underline{\phantom{X}} \otimes \underline{\phantom{X}} \left\lbrace
	\begin{aligned}
		v_{+} & \xrightarrow{\Delta^{0}} v_{+-} + v_{-+} \qquad & v_{\dpl} & \xrightarrow{\Delta^{+2}} 0 \qquad & v_{\dpl} & \xrightarrow{\Delta^{+4}} 0 \\
		v_{-} & \xrightarrow{\Delta^{0}} v_{- -} \qquad & v_{\dm} & \xrightarrow{\Delta^{+2}} 0 \qquad & v_{\dm} & \xrightarrow{\Delta^{+4}} v_{+ +}
	\end{aligned}\right.	
\end{equation*}

\begin{equation*}\label{3Eq:dkkdiff5}
	\Delta : \bullet \rightarrow \bullet \otimes \underline{\phantom{X}} \left\lbrace
	\begin{aligned}
		v_{\dpl} & \xrightarrow{\Delta^{0}} v_{\dpl -} \qquad & v_{\dpl} & \xrightarrow{\Delta^{+2}} v_{\dm +} \qquad & v_{\dpl} & \xrightarrow{\Delta^{+4}} 0 \\
		v_{\dm} & \xrightarrow{\Delta^{0}} v_{\dm -} \qquad & v_{\dm} & \xrightarrow{\Delta^{+2}} 0 \qquad & v_{\dm} & \xrightarrow{\Delta^{+4}} v_{\dpl +}
	\end{aligned}\right.	
\end{equation*}

\begin{equation*}\label{3Eq:dkkdiff6}
	\Delta : \underline{\phantom{X}} \rightarrow \bullet \otimes \bullet \left\lbrace
	\begin{aligned}
		v_{+} & \xrightarrow{\Delta^{0}} v_{\dpl \dm} + v_{\dm \dpl} \qquad & v_{+} & \xrightarrow{\Delta^{+2}} 0 \qquad & v_{+} & \xrightarrow{\Delta^{+4}} 0 \\
		v_{-} & \xrightarrow{\Delta^{0}} 0 \qquad & v_{-} & \xrightarrow{\Delta^{+2}} v_{\dm \dm} \qquad & v_{+} & \xrightarrow{\Delta^{+2}} v_{\dpl \dpl}
	\end{aligned}\right.	
\end{equation*}
\noindent\hrulefill
\begin{equation*}\label{3Eq:dkkdiff7}
	\eta : \underline{\phantom{X}} \rightarrow \underline{\phantom{X}} \left\lbrace
	\begin{aligned}
		v^{\text{u}}_{+} & \xrightarrow{\eta^{0}} v^{\ell}_{+} \qquad & v^{\text{u}}_{\dpl} & \xrightarrow{\eta^{+2}} 0 \qquad & v^{\text{u}}_{\dpl} & \xrightarrow{\eta^{+4}} 0 \\
		v^{\ell}_{+} & \xrightarrow{\eta^{0}} 2 v^{\text{u}}_{-} \qquad & v^{\ell}_{\dpl} & \xrightarrow{\eta^{+2}} 0 \qquad & v^{\ell}_{\dpl} & \xrightarrow{\eta^{+4}} 0 \\
		v^{\text{u}}_{-} & \xrightarrow{\eta^{0}} v^{\ell}_{-} \qquad & v^{\text{u}}_{\dm} & \xrightarrow{\eta^{+2}} 0 \qquad & v^{\text{u}}_{\dm} & \xrightarrow{\eta^{+4}} 0  \\
		v^{\ell}_{-} & \xrightarrow{\eta^{0}} 0 \qquad & v_{\dm \dm} & \xrightarrow{\eta^{+2}} 0 \qquad & v_{\dm \dm} & \xrightarrow{\eta^{+4}} 2 v^{\text{u}}_{\dpl}
	\end{aligned}\right.	
\end{equation*}
The \( \eta \) map changes the superscript globally; for example, \( \eta \otimes \text{id} \left( \vlp \otimes \vlp \right) = 2 \vum \otimes \vup \).

\begin{equation*}\label{3Eq:dkkdiff8}
	\eta : \bullet \rightarrow \bullet \left\lbrace
	\begin{aligned}
		v^{\text{u}}_{\dpl} & \xrightarrow{\eta^{0}} v^{\ell}_{\dpl} \qquad & v^{\text{u}}_{\dpl} & \xrightarrow{\eta^{+2}} 0 \qquad & v^{\text{u}}_{\dpl} & \xrightarrow{\eta^{+4}} 0 \\
		v^{\ell}_{\dpl} & \xrightarrow{\eta^{0}} 0 \qquad & v^{\ell}_{\dpl} & \xrightarrow{\eta^{+2}} 2 v^{\text{u}}_{\dm} \qquad & v^{\ell}_{\dpl} & \xrightarrow{\eta^{+4}} 0 \\
		v^{\text{u}}_{\dm} & \xrightarrow{\eta^{0}} v^{\ell}_{\dm} \qquad & v^{\text{u}}_{\dm} & \xrightarrow{\eta^{+2}} 0 \qquad & v^{\text{u}}_{\dm} & \xrightarrow{\eta^{+4}} 0 \\
		v^{\ell}_{\dm} & \xrightarrow{\eta^{0}} 0 \qquad & v_{\dm \dm} & \xrightarrow{\eta^{+2}} 0 \qquad & v_{\dm \dm} & \xrightarrow{\eta^{+4}} 2 v^{\text{u}}_{\dpl}
	\end{aligned}\right.	
\end{equation*}
\end{minipage}
\caption{Differential components.}\label{3Fig:diffcomp}
\end{figure}

Notice that the maps \( m^0 \), \( \Delta^0 \), and \( \eta^0 \) are all \(j\)- and \(c\)-graded of degree \( 0 \), so that \( \cdkh \left( D, \gamma \right) \) is a trigraded chain complex. The chain homotopy equivalence class of \( \cdkh \left( D, \gamma \right) \) depends only on the link represented by \( D \).

\begin{theorem}[Theorem \(3.4\) of \cite{Rushworth2017b}]
The chain homotopy equivalence class of \( \cdkh \left( D , \gamma \right) \) is an invariant of \( L \), the link represented by \( D \), so that its homology is also. We denote this homology \( \dkh \left( L, \gamma \right) \) and refer to it as the \emph{doubled Khovanov homology of \( L \) with respect to \( \gamma \)}.
\end{theorem}

\subsection{Perturbations}\label{3Subsec:perturb}
For our purposes we do not need the full doubled Khovanov homology of a link, only a perturbation of it. As in the case of classical Khovanov homology we add terms to the differential to produce the desired perturbed theory.

\begin{definition}[Totally reduced homology]\label{3Def:trh}
Let \( D \) be a diagram of an oriented link \( L \hookrightarrow \Sigma_g \times I \). Given \( \gamma \in H^1 \left( \Sigma_g ; \Z_2 \right) \) let \( \cdkh '' \left( D , \gamma \right) \) denote the chain complex whose chain spaces are those of \( \cdkh \left( D , \gamma \right) \) but with an altered differential. This differential is obtained from that of \( \cdkh \left( D , \gamma \right) \) by adding the terms denoted \( m^{+2} \), \( m^{+4} \), \( \Delta^{+2} \), \( \Delta^{+4} \), \( \eta^{+2} \) and \( \eta^{+4} \) in \Cref{3Fig:diffcomp}; we write \( m '' = m^{0} + m^{2} + m^{4} \), and similarly for \( \Delta '' \), \( \eta '' \). The chain complex \( \cdkh '' \left( D , \gamma \right) \) is known as the \emph{totally reduced complex of \( D \) with respect to \( \gamma \)}.  
\end{definition}

Notice that the maps \( m^{+2} \), \( \Delta^{+2} \), and \( \eta^{+2} \) are \(j\)-graded of degree \(0\), and \(c\)-graded of degree \( +2 \). The maps \( m^{+4} \), \( \Delta^{+4} \), and \( \eta^{+4} \) \(j\)-graded of degree \(+4\) and \(c\)-graded of degree \(0\). It follows that \( \cdkh '' \left( D , \gamma \right) \) is filtered in both the \(j\)-\ and \(c\)-gradings. We abuse notation and denote by \( j \) and \( c \) the induced filtration gradings on \( \cdkh '' \left( D , \gamma \right) \).

\begin{theorem}[Theorem \(3.8\) of \cite{Rushworth2017b}]\label{3Thm:trhinv}
Let \( D \) be a diagram of an oriented link \( L \hookrightarrow \Sigma_g \times I \). The chain homotopy equivalence class of \( \cdkh '' \left( D , \gamma \right) \) is an invariant of \( L \), so that its homology is also. This homology is denoted \( \dkh '' \left( L , \gamma \right) \), and known as the \emph{totally reduced homology of \( L \) with respect to \( \gamma \)}.
\end{theorem}

In \cite[Section \(3\)]{Rushworth2017} a homology theory of virtual links is constructed, analogous to the Lee homology of classical links. A virtual link is an equivalence class of links in thickened surfaces, up to self-diffeomorphism of the surface and certain permitted handle additions. As such, the homology theory constructed in \cite{Rushworth2017} descends to a well-defined homology theory of links in thickened surfaces. Given an oriented link \( L \hookrightarrow \Sigma_g \times I \) we denote by \( \dkh ' ( L ) \) the \emph{doubled Lee homology of \( L \)}. For full details see \cite[Section \(3\)]{Rushworth2017}.

\begin{proposition}\label{3Prop:leeiso}
Forgetting the \(c\)-grading, \( \dkh '' \left( L , \gamma \right) \) is isomorphic to \( \dkh ' ( L ) \).
\end{proposition}
\begin{proof}
Compare the differential components of doubled Lee homology, given in \cite[Definition \(3.1\)]{Rushworth2017}, to those of the totally reduced homology (given in \Cref{3Fig:diffcomp}). Also notice that the \(j\)-grading, as defined in \Cref{3Def:gradings}, does not depend on the dotting with respect to \( \gamma \).
\end{proof}

\begin{remark}
In spite of \Cref{3Prop:leeiso}, the totally reduced homology is not an invariant of virtual links. Specifically, the \(c\)-grading is not invariant under self-diffeomorphism of the surface or the permitted handle additions.
\end{remark}

In \cite[Section \(3\)]{Rushworth2017} distinguished generators of doubled Lee homology are described, which yield generators of the totally reduced homology by \Cref{3Prop:leeiso}. These generators come in quadruples; for the remainder of this work \( x^{\text{u}}, \overline{x}^{\text{u}}, x^{\ell}, \overline{x}^{\ell} \) shall denote such a quadruple.

\begin{remark}
The reader familiar with the Lee homology of classical links will recall that there are distinguished generators, \( \sg \), \(\overline{\sg}\), corresponding to alternately colourable smoothings of the argument diagram. The generators \( x^{\text{u/l}} \) above are \( x^{\text{u/l}} = \sg \pm \overline{\sg} \), and \( \overline{x}^{\text{u/l}} = \sg \mp \overline{\sg} \).
\end{remark}

\subsection{Functoriality}\label{3Subsec:func}
Although the totally reduced homology is not constructed using a TQFT, it is functorial with respect to link cobordism. For full details see \cite[Section \(3.2\)]{Rushworth2017} and \cite[Section \(3.4\)]{Rushworth2017b}.

\begin{definition}\label{3Def:cobmap}
Let \( \left( S, \Sigma_g \times I \right) \) be a strict concordance between oriented links \( L_1 \hookrightarrow \Sigma_{g} \times I \) and \( L_2 \hookrightarrow \Sigma_{g} \times I \). There is a map \( \phi_S : \dkh '' \left( L_1 , \gamma \right) \rightarrow \dkh '' \left( L_2 , \gamma \right) \) induced by \( \left( S, \Sigma_g \times I \right) \), for all \( \gamma \in H^1 \left( \Sigma_g ; \Z_2 \right) \).

Further, if \( \left( S , M \right) \) is a cobordism between \( L_1 \hookrightarrow \Sigma_{g_1} \times I \) and \( L_2 \hookrightarrow \Sigma_{g_2} \times I \), there is a map \( \widetilde{\phi}_S : \dkh ' ( L_1 ) \rightarrow \dkh ' ( L_2 ) \).
\end{definition}
Similar to the case of the cobordism maps on Lee homology, \( \phi_S \) is filtered of \(j\)-degree \( \chi ( S ) \), and \( c \)-degree \( \frac{1}{2} \chi ( S ) \).

By construction, the map assigned to a cobordism factors through the maps assigned to cobordisms it may be decomposed into.
\begin{proposition}\label{3Prop:factoring}
Let \( \left( S_1 , M_1 \right) \) be a cobordism from \( L_1 \hookrightarrow \Sigma_{g_1} \times I \) to \( L_2 \hookrightarrow \Sigma_{g_2} \times I \), and \( \left( S_2 , M_2 \right) \) a cobordism from \( L_2 \hookrightarrow \Sigma_{g_2} \times I \) to \( L_3 \hookrightarrow \Sigma_{g_3} \times I \). Then \( ( S , M ) = \left( S_1 \cup S_2 , M_1 \cup M_2 \right) \) is a cobordism from \( L_1 \) to \( L_3 \), and \( \widetilde{\phi}_S = \widetilde{\phi}_{S_2} \circ \widetilde{\phi}_{S_1} \). If \( M_1 = M_2 = \Sigma_{g_1} \times I \) then \( \phi_S = \phi_{S_2} \circ \phi_{S_1} \) also. 
\end{proposition}

Concordances induce isomorphisms on doubled Lee homology. If a concordance is strict, it induces an isomorphism on the totally reduced homology.
\begin{theorem}[Theorem \(3.21\) of \cite{Rushworth2017}, Proposition \(3.15\) of \cite{Rushworth2017b}]\label{3Thm:isos}
If \( \left( S , M \right) \) is a concordance then \( \widetilde{\phi}_S \) is an isomorphism. If \( M = \Sigma_g \times I \) then \( \phi_S \) is an isomorphism also.
\end{theorem}

In addition to strict cobordisms, the totally reduced homology enjoys functoriality with respect to pseudostrict cobordisms. To establish this, we show that if \( ( S, M ) \) is a pseudostrict cobordism, passing a critical point of \( M \) does not affect the totally reduced homology. We may then concatenate the maps assigned to the strict pieces of \( ( S, M ) \).
\begin{proposition}\label{3Prop:pseudoiso}
Let \( L_1 \hookrightarrow \Sigma_{g} \times I \) and \( L_2 \hookrightarrow \Sigma^{\prime}_{g} \times I \) be links and \( ( S , M ) \) a pseudostrict concordance between them, such that \( S \) is a product cobordism and \( M \) contains exactly one critical point. Then there is an isomorphism
\begin{equation}\label{3Eq:homiso}
	f : H^1 \left( \Sigma_{g} ; \Z_2 \right) \longrightarrow H^1 \left( \Sigma^{\prime}_{g} ; \Z_2 \right)
\end{equation}
and the chain complexes \( \cdkh '' ( D_1 , \gamma ) \) and \( \cdkh '' ( D_2 , f ( \gamma ) ) \) are identical for all diagrams \( D_1 \), \( D_2 \) of \( L_1 \) and \( L_2 \). Thus \( \dkh '' ( L_1 , \gamma ) \) and \( \dkh '' ( L_2 , f ( \gamma ) ) \) are identical also.
\end{proposition}

\begin{proof}
If the critical point of \( M \) is of index \( 0 \) or \( 3 \) the result is clear from the construction of the totally reduced homology: the number of circles in a smoothing, the edges of the cube of resolutions, and the dotting are unchanged.

Suppose that the critical point of of index \( 1 \) (the index \( 2 \) case is obtained by reversing the cobordism and applying the following proof). As \( ( S , M ) \) is pseudostrict, the attaching sphere of the handle corresponding to the critical point must be a separating curve; denote it by \( \sigma \). Distinguishing this curve induces a direct sum decomposition of \( H^1 \left( \Sigma_{g} ; \Z_2 \right) \) as follows. Let \( \Sigma_g = F_1 \cup_{\sigma} F_2 \), where \( F_1 \), \( F_2 \) are compact orientable surfaces with boundary \( S^1 \). Denote by \( \Sigma_{g_i} \) the result of collapsing \( F_i \) to a point. Notice that \( \Sigma^{\prime}_{g} = \Sigma_{g_1} \sqcup \Sigma_{g_2} \). We have
\begin{equation*}
	\begin{aligned}
		H^1 \left( \Sigma_{g} ; \Z_2 \right) &\cong H^1 \left( \Sigma_{g_1} ; \Z_2 \right) \oplus H^1 \left( \Sigma_{g_2} ; \Z_2 \right) \\
		&\cong H^1 \left( \Sigma^{\prime}_{g} ; \Z_2 \right)
	\end{aligned}
\end{equation*}
Denote the isomorphism described by \( f \). It is clear that the number of circles in a smoothing and the edges of the cube of resolutions are unchanged, and that the dotting with respect to \( \gamma \) and \( f ( \gamma ) \) are equivalent.
\end{proof}

\begin{definition}\label{3Def:psmap}
Let \( L_1 \hookrightarrow \Sigma_{g} \times I \) and \( L_2 \hookrightarrow \Sigma^{\prime}_{g} \times I \) be links and \( ( S , M ) \) a pseudostrict cobordism between them. There is a map \( \phi_S : \dkh '' \left( L_1 , \gamma \right) \rightarrow \dkh '' \left( L_2 , \gamma \right) \) induced by \( \left( S, \Sigma_g \times I \right) \), for all \( \gamma \in H^1 \left( \Sigma_g ; \Z_2 \right) \) (we have suppressed the notation \( f ( \gamma ) \) of \Cref{3Eq:homiso}). This map is defined by splitting \( ( S , M ) \) into strict pieces, and concatenating the maps assigned to these pieces using the identification of their domains and codomains given by \Cref{3Prop:pseudoiso}.
\end{definition}
The map assigned to a pseudostrict cobordism enjoys the factoring property, described in \Cref{3Prop:factoring}, by construction. In addition, if \( ( S , M ) \) is a pseudostrict cobordism then \( \phi_S \) is filtered of \(j\)-degree \( \chi ( S ) \), and \( c \)-degree \( \frac{1}{2} \chi ( S ) \).

The following is a corollary to \Cref{3Thm:isos}.
\begin{corollary}\label{3Cor:iso}
If \( ( S , M ) \) is a pseudostrict concordance then \( \phi_S \) is an isomorphism.
\end{corollary}

\subsection{Properties}\label{3Subsec:properties}
We conclude this section by determining properties of the totally reduced homology we require in \Cref{4Sec:ascent}.

Vertical annuli within \( \Sigma_g \times I \) represent available handle destabilizations, or equivalently index \( 1 \) Morse critical points within cobordisms. Given \( L \hookrightarrow \Sigma_g \times I \), suppose that there exists a vertical annulus \( \sigma \times I \) in \( \Sigma_g \times I \), such that \( L \cap ( \sigma \times I ) = \emptyset \). If \( \sigma \) represents the Poincar\'e dual to \( \gamma \), then the \(c\)-grading of \( \dkh '' ( L , \gamma ) \) is determined by the \( j \)-grading.
\begin{lemma}\label{3Lem:missing}
Let \( L \hookrightarrow \Sigma_g \times I \) be a link and \( \gamma \in H^1 \left( \Sigma_g ; \Z_2 \right) \). Suppose there exists a diagram, \( D \), of \( L \) and a simple closed curve \( \sigma \) on \( \Sigma_g \), representing (the Poincar\'e dual to) \( \gamma \), with \( D \cap \sigma = \emptyset \). Then
\begin{equation*}
	c ( x^{\text{u/l}} ) = \dfrac{1}{2} j ( x^{\text{u/l}} )
\end{equation*}
for all \( x^{\text{u/l}} \in \dkh '' \left( L , \gamma \right) \).
\end{lemma}
\begin{proof}
If \( D \cap \sigma = \emptyset \) then no circles within \( \llbracket D , \gamma \rrbracket \) acquire dots with respect to \( \gamma \). The result is then clear from \Cref{3Eq:dottedgrading}.
\end{proof}  

The generators of classical Lee homology come in pairs, and the grading of one is prescribed by the grading of the other. In the case of the totally reduced homology the generators come in quadruples, and the grading of any one of them prescribes the gradings of the others.
\begin{lemma}\label{3Lem:gradingshift}
We have
\begin{equation*}
	\begin{aligned}
		j ( x^{\text{u}} ) &= j ( x^{\ell} ) + 1 \\
		c ( x^{\text{u}} ) &= c ( x^{\ell} ) + \dfrac{1}{2} 
	\end{aligned}
\end{equation*}
for all \( x^{\text{u/l}} \in \dkh '' \left( L , \gamma \right) \).
\end{lemma}

\begin{proof}
In \cite[Proof of Theorem \(3.5\)]{Rushworth2017} Gaussian elimination is used to produce a complex that is chain homotopy equivalent to \( \cdkh '' \left( D , \gamma \right) \), with chain spaces spanned by the set of all \( x^{\text{u/l}}, \overline{x}^{\text{u/l}} \), and with vanishing differential. Therefore the new complex splits as a direct sum of upper and lower terms, and one may obtain the gradings of the shifted part of its homology from the unshifted part via \Cref{3Eq:quantumgrading,3Eq:dottedgrading}. It follows that \( \dkh '' \left( D , \gamma \right) \) splits likewise.
\end{proof}

\begin{lemma}\label{3Lem:gradingbar}
We have
\begin{equation*}
	\begin{aligned}
		j ( x^{\text{u/l}} ) &= j ( \overline{x}^{\text{u/l}} ) \pm 2 \\
		c ( x^{\text{u/l}} ) &= c ( \overline{x}^{\text{u/l}} ) \pm 1 
	\end{aligned}
\end{equation*}
for all \( x^{\text{u/l}} \in \dkh '' \left( L , \gamma \right) \). The \( \pm \) signs are independent.
\end{lemma}

The proof of the \( j \)-grading statement follows from \Cref{3Prop:leeiso} and \cite[Lemma \(4.2\)]{Rushworth2017}, while proof of the \(c\)-grading statement is essentially identical to that of \cite[Lemma \(4.2\)]{Rushworth2017} (see also \cite[Lemma \( 3.5 \)]{Rasmussen2010}).

The final property we require concerns the interaction between the maps assigned to cobordisms and the upper and lower superscripts.
\begin{lemma}\label{3Lem:upperlower}
Let \( \widetilde{\phi}_S : \dkh ' ( L_1 ) \rightarrow \dkh ' ( L_2 ) \) be a map induced by a cobordism. Then
	\begin{equation*}
		\widetilde{\phi}_S ( x^{\text{u}} ) = y^{\text{u/l}} \Leftrightarrow \widetilde{\phi} ( x^{\ell} ) = y^{\text{l/u}}
	\end{equation*}
up to a nonzero scalar. A map induced by a strict cobordism \( \phi_S : \dkh '' ( L_1, \gamma ) \rightarrow \dkh '' ( L_2, \gamma ) \) behaves similarly.
\end{lemma}
\begin{proof}
There is a convenient basis of \( \cdkh '' ( L , \gamma ) \), first given in the case of Lee homology by Bar-Natan and Morrison \cite{Bar-Natan2006}. Let \( \lbrace r, g \rbrace \) be the basis of \( \mathcal{A} \) where
\begin{equation*}
	\begin{aligned}
		r = \dfrac{v_+ + v_-}{2} \\
		g = \dfrac{v_+ - v_-}{2}
	\end{aligned}
\end{equation*}
and similarly for \( \overset{\bullet}{r} \), \( \overset{\bullet}{g} \). We denote the corresponding generators of \( \mathcal{A} \oplus \left( \mathcal{A}\lbrace -1 \rbrace \right) \) as \( r^\text{u} \), \( r^\ell \), \( g^\text{u} \), and \( g^\ell \). The distinguished generators \( x^{\text{u/l}}, \overline{x}^{\text{u/l}} \) are expressed in this basis. The maps \( m'' \), \( \Delta '' \), and \( \eta '' \) have the following form, regardless of the dotting of the argument:
\begin{equation}\label{3Eq:rgbasis}
	\begin{aligned}
		m'' ( r \otimes r ) &= r \qquad &\Delta'' ( r ) &= 2r \otimes r \qquad &\eta'' ( r^\text{u}) &= r^\ell \\
		m'' ( g \otimes g ) &= g \qquad &\Delta'' ( g ) &= -2g \otimes g \qquad &\eta'' ( g^\text{u}) &= g^\ell \\
		m'' ( r \otimes g ) &= m'' ( g \otimes r ) = 0 & & \qquad &\eta'' ( r^\ell) &= 2r^\text{u} \\
		& & & &\eta'' ( g^\ell) &= -2g^\text{u}
	\end{aligned}
\end{equation}
As described in \cite[Section \(3.2\)]{Rushworth2017} and \cite[Section \(3.4\)]{Rushworth2017b} the map \( \widetilde{\phi} \) acts as a composition of the maps assigned to elementary cobordisms. These elementary cobordism maps act as either \( m '' \), \( \Delta '' \) or \( \eta '' \). Let \( f \in \lbrace m '' , \Delta '' , \eta '' \rbrace \): it is clear from \Cref{3Eq:rgbasis} that \( f ( x^{\text{u}} ) = y^{\text{u/l}} \Leftrightarrow f ( x^{\ell} ) = y^{\text{l/u}} \), up to a nonzero scalar. As \( \widetilde{\phi} \) is a composition of such maps, the result follows.
\end{proof}

\section{Detecting ascent concordance}\label{4Sec:ascent}
In \Cref{4Subsec:obstruction} we demonstrate that the totally reduced homology may be used to obstruct descent concordance, and provide examples of ascent concordant links in \Cref{4Subsec:examples}.

\subsection{Totally reduced homology obstructs descent concordance}\label{4Subsec:obstruction}
First, we define the \emph{totally nontrivial} property of links. As \Cref{4Prop:attaching} shows, a totally nontrivial link must intersect the attaching sphere of a destabilizing handle. Next, in \Cref{4Prop:conctnt,4Prop:gen0tnt} we verify that a totally nontrivial link cannot be made disjoint to an attaching sphere of a destabilizing handle at any stage of a pseudostrict concordance. Destabilizing handles correspond to index \( 1 \) Morse critical points, and it follows that totally nontrivial links cannot be destabilized, up to pseudostrict concordance.

\begin{definition}[Totally nontrivial]\label{4Def:tnt}
An oriented link \( L \hookrightarrow \Sigma_g \times I \) is \emph{totally nontrivial} if for every non-identity element \( \gamma \in H^1 \left( \Sigma_g ; \Z_2 \right) \) there exists \( x \in \dkh '' \left( D, \gamma \right) \) such that
\begin{equation*}
c ( x ) \neq \dfrac{1}{2} j ( x ).
\end{equation*}
\end{definition}

A totally nontrivial link must intersect the attaching spheres of destabilizing handles.
\begin{lemma}\label{4Prop:attaching}
Let \( \sigma \) be the attaching sphere of a destabilizing handle on \( \Sigma_g \). If \( L \hookrightarrow \Sigma_g \times I \) is totally nontrivial, then \( D \cap \sigma \neq \emptyset \) for all diagrams \( D \) of \( L \).
\end{lemma}

\begin{proof}
Suppose \( \sigma \) represents (the Poincar\'e dual to)  \( \gamma \in H^1 \left( \Sigma_g ; \Z_2 \right) \). As \( L \) is totally nontrivial, \( \dkh '' \left( L , \gamma \right) \) must contain an element \( x \) such that
\begin{equation*}
c ( x ) \neq \dfrac{1}{2} j ( x ).
\end{equation*}
The result then follows from the contrapositive to \Cref{3Lem:missing}.
\end{proof}
It follows that if \( L \) is totally nontrivial and \( \sigma \) is the attaching sphere of destabilizing handle, then \( L \cap ( \sigma \times I ) \neq \emptyset \). Thus a totally nontrivial link does not support a destabilizing handle. We are interested in concordance, however, and therefore must verify that the totally nontrivial property interacts well with particular concordances.

\begin{proposition}\label{4Prop:conctnt}
Let \( L_1 \) and \( L_2 \) be strictly concordant links. Then \( L_1 \) is totally nontrivial if and only if \( L_2 \) is totally nontrivial.
\end{proposition}
\begin{proof}
Invoke the isomorphism on the totally reduced homologies induced by a strict concordance, as stated in \Cref{3Thm:isos}.
\end{proof}

\Cref{4Prop:conctnt} is not enough to obstruct descent concordance, however. To see this, let \( ( S, M ) \) be a concordance with initial link \( L \hookrightarrow \Sigma_g \times I \). Traverse the concordance, starting at \( L \), until a critical point corresponding to a (de)stabilizing handle is met. Cutting open \( ( S, M ) \) immediately before this critical point yields \( ( S' , \Sigma_g \times I ) \), a pseudostrict genus \( 0 \) cobordism (not necessarily a concordance). It follows that a destabilization may occur at a link that is merely pseudostrictly genus \( 0 \) cobordant to \( L \). 

As such, we must verify that the totally nontrivial property obstructs index \(1\) critical points within pseudostrict genus \( 0 \) cobordisms obtained by cutting open concordances. To ease exposition we prove the case of strict genus \( 0 \) cobordisms; the pseudostrict case follows identically.
\begin{proposition}\label{4Prop:gen0tnt}
Let \( ( S_1 , \Sigma_{g_1} \times I ) \) be a strict genus \( 0 \) cobordism from \( L_1 \hookrightarrow \Sigma_{g_1} \times I \) to \( J \hookrightarrow \Sigma_{g_1} \times I \), and \( ( S_2 , M ) \) a cobordism from \( J \hookrightarrow \Sigma_{g_1} \times I \) to \( L_2 \hookrightarrow \Sigma_{g_2} \times I \). Suppose that \( ( S, M' ) = ( S_1 \cup_J S_2 , \Sigma_{g_1} \times I \cup M ) \) is a concordance from \( L_1 \hookrightarrow \Sigma_{g_1} \times I \) to \( L_2 \hookrightarrow \Sigma_{g_2} \times I \).

If \( L_1 \) is totally nontrivial and \( \sigma \) is the attaching sphere of a destabilizing handle on \( \Sigma_{g_1} \), then \( J \cap ( \sigma \times I ) \neq \emptyset \).
\end{proposition}
This proposition requires the assumption that the genus \( 0 \) cobordism is obtained by cutting open a concordance: there exist totally nontrivial links that are genus \( 0 \) cobordant to the unknot (the relevant genus \( 0 \) cobordism may not appear within a concordance, therefore).

\begin{proof}[Proof of \Cref{4Prop:gen0tnt}]
The map \( \widetilde{\phi}_S : \dkh ' ( L_1 ) \rightarrow \dkh '' ( L_2 ) \) is an isomorphism by \Cref{3Thm:isos}, and \( \widetilde{\phi}_S = \widetilde{\phi}_{S_2} \circ \widetilde{\phi_{S_1}} \) by \Cref{3Prop:factoring}. Therefore \( \widetilde{\phi_{S_1}} \) is injective, and so is \( \phi_{S_1} \) by \Cref{3Prop:leeiso}.

We require a fact regarding \( \phi_{S_1} \) and the upper/lower superscripts. Suppose that
\begin{equation}\label{4Eq:phiupper}
	\begin{aligned}
		j ( \phi_{S_1} ( x^{\text{u}} ) ) &= j ( x^{\text{u}} ) + m_1 \\
		c ( \phi_{S_1} ( x^{\text{u}} ) ) &= c ( x^{\text{u}} ) + m_2.
	\end{aligned}
\end{equation}
(Note that the injectivity of \( \phi_{S_1} \) guarantees that \( \phi_{S_1} ( x^{\text{u/l}} ) \neq 0 \).) We claim that
\begin{equation}\label{4Eq:philower}
	\begin{aligned}
		j ( \phi_{S_1} ( x^{\ell} ) ) &= j ( x^{\ell} ) + m_1 \\
		c ( \phi_{S_1} ( x^{\ell} ) ) &= c ( x^{\ell} ) + m_2 
	\end{aligned}
\end{equation}
also. To see this, apply \Cref{3Lem:gradingshift} and \Cref{3Lem:upperlower} to \Cref{4Eq:phiupper} to obtain
\begin{equation*}
		j ( \phi_{S_1} ( x^{\ell} ) ) \pm 1 = j ( x^{\ell} ) + 1 + m_1
\end{equation*}
so that \( j ( \phi_{S_1} ( x^{\ell} ) ) = j ( x^{\ell} ) \mp 1 + 1 + m_1\). Recall that the differential of the totally reduced homology consists of a component of \(j\)-degree \( 0 \) and another of \(j\)-degree \( +4 \), and that \( \phi_{S_1} \) is \(j\)-graded (\(c\)-graded) of degree \( \chi ( S_1 ) \) (\( \frac{1}{2} \chi (S_1) \). Therefore we have \( m_1 = \chi (S_1) \mod 4 \). If \( j ( \phi_{S_1} ( x^{\ell} ) ) = j ( x^{\ell} ) + 2 + m_1 \), then \( 2 + m_1 = \chi (S_1) \mod 4 \) also, yielding a contradiction. The argument for the \(c\)-grading statement is essentially identical.

We now verify the proposition. Assume towards a contradiction that there exists \( \sigma \), the attaching sphere of a destabilizing handle, such that \( J \cap ( \sigma \times I ) = \emptyset \). Therefore there exists \( D \), a diagram of \( J \), such that \( D \cap \sigma = \emptyset \). Let \( \sigma \) represent the Poincar\'e dual to \( \gamma \in H^1 \left( \Sigma_g ; \Z_2 \right) \). In what follows we shall continue to denote the induced filtration gradings on \( \dkh '' ( J , \gamma ) \) by \( j \) and \( c \), and denote the honest gradings on \( \cdkh '' ( D , \gamma ) \) by \( \widetilde{j} \) and \( \widetilde{c} \).

As \( D \cap \sigma = \emptyset \), no smoothings within \( \llbracket D , \gamma \rrbracket \) acquire dots, and we have
\begin{equation}\label{4Eq:g1}
	\begin{aligned}
		\widetilde{c} ( y ) &= \dfrac{1}{2} \widetilde{j} ( y ) \\
		c ( [y] ) &= \dfrac{1}{2} j ( [y] )
	\end{aligned}
\end{equation}
for all \( y \in \cdkh '' ( D , \gamma ) \) by \Cref{3Eq:dottedgrading}.

As \( L_1 \) is totally nontrivial there exists an \( x^{\text{u/l}} \in \dkh '' \left( L , \gamma \right) \) such that \( c ( x^{\text{u/l}} ) \neq \frac{1}{2} j ( x^{\text{u/l}} ) \). Suppose that
\begin{equation}\label{4Eq:g2}
	c ( x^{\text{u/l}} ) = k j ( x^{\text{u/l}} ), \quad k \in \Z \left[ \frac{1}{2} \right] .
\end{equation}
Without loss of generality we may assume that \( j ( x^{\text{u/l}} ) \neq 0 \): if \( j ( x^{\text{u/l}} ) = 0 \) it follows from \Cref{3Lem:gradingshift,3Lem:gradingbar} that at least one of \( \lbrace x^{\text{l/u}}, \overline{x}^{\text{u}}, \overline{x}^{\ell} \rbrace \) is as desired. Via \Cref{4Eq:phiupper,4Eq:philower,4Eq:g1,4Eq:g2} we obtain
\begin{equation*}
	\begin{aligned}
		c ( \phi_{S_1} ( x^{\text{u/l}} ) ) &= c ( x^{\text{u/l}} ) + m_2 \\
		\dfrac{1}{2} j ( \phi_{S_1} ( x^{\text{u/l}} ) ) &= k j ( x^{\text{u/l}} ) + m_2 \\
		\dfrac{1}{2} \left( j ( x^{\text{u/l}} ) + m_1 \right) &= k j ( x^{\text{u/l}} ) + m_2
	\end{aligned}
\end{equation*}
so that
\begin{equation}\label{4Eq:p1}
	\dfrac{1}{2} m_1 - m_2 = \left( k - \dfrac{1}{2} \right) j ( x^{\text{u/l}} ).
\end{equation}

Suppose \( y_1, y_2 \in \cdkh '' ( D , \gamma )  \) are homologous and
\begin{equation*}
	\begin{aligned}
		\widetilde{j} ( y_1 ) &\geq \widetilde{j} ( y_2 ) \\
		\widetilde{c} ( y_1 ) &\leq \widetilde{c} ( y_2 ).
	\end{aligned}
\end{equation*}
But \( \widetilde{c} ( y_i ) = \dfrac{1}{2} \widetilde{j} ( y_i ) \) so that \( \widetilde{c} ( y_1 ) = \widetilde{c} ( y_2 ) \), and \( 
\widetilde{j} ( y_1 ) = \widetilde{j} ( y_2 ) \). Thus the \( j \)- and \( c \)-gradings may be realised on one element of a homology class. Further, we have
\begin{equation*}
	\begin{aligned}
		\widetilde{j} ( y_1 ) &= \widetilde{j} ( y_2 ) + 4 r_1 \\
		\widetilde{c} ( y_1 ) &= \widetilde{c} ( y_2 ) + 2 r_2
	\end{aligned}
\end{equation*}
for \( r_1, r_2 \in \N \), due to the form of the differential of \( \cdkh '' ( D , \gamma ) \). \Cref{4Eq:g1} implies that
\begin{equation*}
	\begin{aligned}
		\widetilde{j} ( y_1 ) &= \widetilde{j} ( y_2 ) + 4 r_1 \\
		2\widetilde{c} ( y_1 ) &= 2\widetilde{c} ( y_2 ) + 4 r_1 \\
		\widetilde{c} ( y_1 ) &= \widetilde{c} ( y_2 ) + 2 r_1 \\
	\end{aligned}
\end{equation*}
so that \( r_1 = r_2 \). It follows that
\begin{equation*}
	\begin{aligned}
		m_1 &= \chi ( S ) + 4t \\
		m_2 &= \dfrac{1}{2} \chi ( S ) + 2t
	\end{aligned}
\end{equation*}
for \( m_1 \) and \( m_2 \) as given in \Cref{4Eq:phiupper,4Eq:philower}, and \( t \in \N \). Then
\begin{equation}\label{4Eq:p2}
	\dfrac{1}{2} m_1 - m_2 = \dfrac{1}{2} \chi ( S ) + 2t - \dfrac{1}{2} \chi ( S ) - 2t = 0.
\end{equation}
Combining \Cref{4Eq:p1,4Eq:p2}, and recalling that \( j ( x^{\text{u/l}} ) \neq 0 \), we obtain \( k - \frac{1}{2} = 0 \), a contradiction.
\end{proof}

\begin{proposition}\label{4Prop:psgen0tnt}
Let \( ( S_1 , \Sigma_{g_1} \times I ) \) be a pseudostrict genus \( 0 \) cobordism from \( L_1 \hookrightarrow \Sigma_{g_1} \times I \) to \( J \hookrightarrow \Sigma_{g_1} \times I \), and \( ( S_2 , M ) \) a cobordism from \( J \hookrightarrow \Sigma_{g_1} \times I \) to \( L_2 \hookrightarrow \Sigma_{g_2} \times I \). Suppose that \( ( S, M' ) = ( S_1 \cup_J S_2 , \Sigma_{g_1} \times I \cup M ) \) is a concordance from \( L_1 \hookrightarrow \Sigma_{g_1} \times I \) to \( L_2 \hookrightarrow \Sigma_{g_2} \times I \).

If \( L_1 \) is totally nontrivial and \( \sigma \) is the attaching sphere of a destabilizing handle on \( \Sigma_{g_1} \), then \( J \cap ( \sigma \times I ) \neq \emptyset \).
\end{proposition}

\begin{proof}
The proof is almost identical to that of \Cref{4Prop:gen0tnt}: simply replace the map assigned to a strict concordance with that assigned to a pseudostrict concordance, given in \Cref{3Def:psmap}.
\end{proof}

With the case of pseudostrict genus \( 0 \) cobordisms complete we can obstruct descent concordance.
\begin{theorem}\label{4Thm:ascent}
Let \( \left( S , M \right) \) be a concordance from \( L_1 \hookrightarrow \Sigma_{g_1} \times I \) to \( L_2 \hookrightarrow \Sigma_{g_2} \times I \). Suppose that \( L_1 \) is totally nontrivial and either
\begin{enumerate}[(i)]
\item \( g_1 > g_2 \)
\end{enumerate}
or
\begin{enumerate}[(i), start=2]
\item \( g_1 = g_2 \) and \( \left( S , M \right) \) is not pseudostrict.
\end{enumerate}
Then \( ( S , M ) \) is ascent.
\end{theorem}
\begin{proof}
We prove Case (ii). Let \( f \) be a Morse function as in \Cref{2Def:exceeding}. As \( \left( S , M \right) \) is not pseudostrict, when traversing the concordance we must encounter a critical point corresponding to a (de)stabilizing handle addition. Let \( p \) be the first such critical point met. Assume towards a contradiction that \( p \) is a destabilizing index \(1\) critical point. At \( p \) a \(3\)-dimensional \(2\)-handle attachment occurs. Let \( \Sigma_{g} = f^{-1} \left( f(p) + \epsilon \right) \) be the level surface to which this handle is attached, and \( \sigma \) the attaching sphere (a nonseparating simple closed curve on \( \Sigma_{g} \)).

The intersection \( \left( \Sigma_{g} \times I \right) \cap S \) is a link in \( \Sigma_{g} \times I \), denoted \( J \). The link \( J \) satisfies the hypothesis of \Cref{4Prop:psgen0tnt}, so that \( J \cap (\sigma \times I) \neq \emptyset \). But if this intersection is non-empty then \( S \) is not smoothly embedded in \( M \), yielding a contradiction. It follows that the critical point \( p \) must be of index \( 2 \), and correspond to a stabilizing handle addition. Thus the genus of level surfaces appearing immediately after it is \( g_1 + 1 \). As \( g_1 + 1 > g_1 = g_2 \) the critical point \( p \) is exceeding, and \( \left( S , M \right) \) is ascent.

The proof for Case (i) is identical, as the fact that \( g_1 > g_2 \) guarantees that \( \left( S , M \right) \) is not pseudostrict.
\end{proof}

\begin{corollary}
Let \( L \hookrightarrow \Sigma_g \times I \) and \( L' \hookrightarrow \Sigma_g \times I \) be concordant and \( L \) totally nontrivial. If \( \dkh \left( L, \gamma \right) \neq \dkh \left( L', \gamma \right) \) for some \( \gamma \in H^1 \left( \Sigma_g ; \Z_2 \right) \), then \( L \) and \( L ' \) are ascent concordant.
\end{corollary} 
\begin{proof}
If \( \dkh \left( L, \gamma \right) \neq \dkh \left( L', \gamma \right) \) then \( L \) and \( L ' \) are not pseudostrictly concordant by \Cref{3Cor:iso}, and by \Cref{4Thm:ascent} any concordance between them is ascent.
\end{proof}

\subsection{Examples}\label{4Subsec:examples}
In this section we prove the theorem stated on \cpageref{1Thm:main}, presenting infinite families of ascent concordant links. First, we present a pair of links whose descent concordance may be obstructed using the totally reduced homology, or by an elementary method. Next, we present a pair of links to which this elementary method does not apply, necessitating the use of the totally reduced homology. We conclude by presenting further examples of ascent concordant links not amenable to the elementary method. Throughout this section we shall denote by \( \lbrace \alpha, \beta \rbrace \) a basis of \( H^1 ( \Sigma_1 ; \Z_2 ) \).

\subsubsection{First family}\label{Sec:1fam}
\begin{figure}
\includegraphics[scale=0.6]{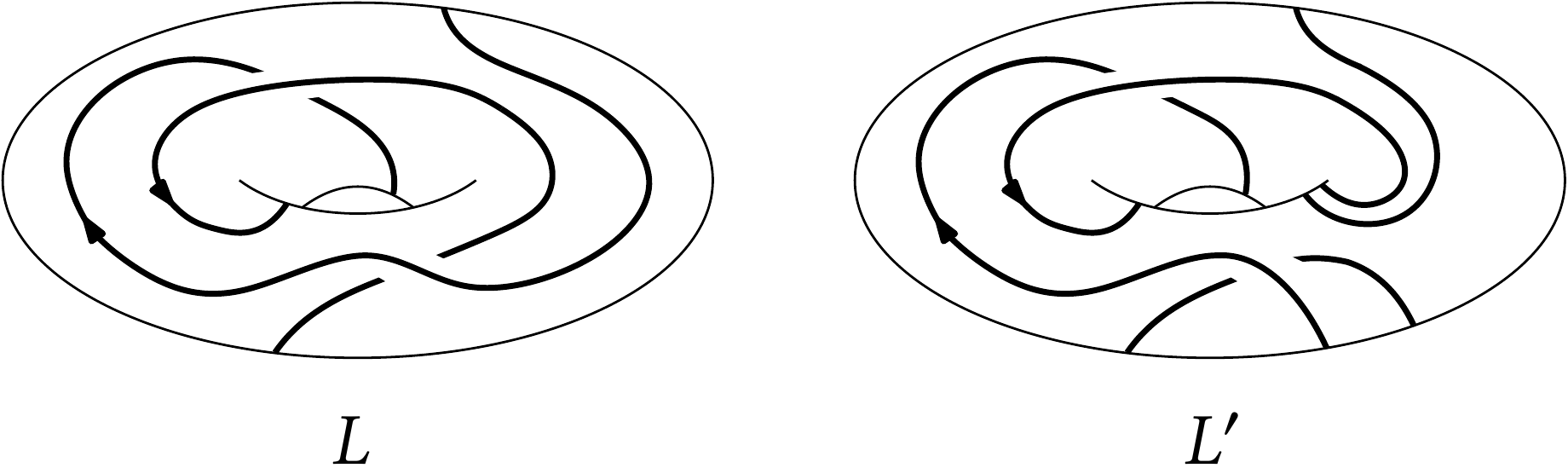}
\caption{A pair of ascent concordant links.}
\label{4Fig:links1}
\end{figure}

Consider the two-component links \( L \hookrightarrow \Sigma_1 \times I \) and \( L ' \hookrightarrow \Sigma_1 \times I \) given in \Cref{4Fig:links1}. The link \( L \) is totally nontrivial: the homologies \( \dkh '' ( L , \gamma ) \) possess a generator of \( (j , c ) \)-bidegree \( (-3, \frac{1}{2} ) \) for all \( \gamma \in \lbrace \alpha, \beta, \alpha + \beta \rbrace \).

Observe that \( L' \) is obtained from \( L \) via a Dehn twist; denote this twist \( \psi \). We may realise \( \psi \) as an ascent concordance, as described in \Cref{4Fig:movie}. The diagram labelled (1) is obtained from \( L \) via an isotopy, then:
\begin{enumerate}[align=left]
\item[(1) to (2):] Add an empty handle (pass an exceeding index \(2\) critical point).

\item[(2) to (3):] Slide the foot of the leftmost handle over the rightmost, via the red path.

\item[(3) to (4):] Destabilize along the blue curve.
\end{enumerate}
The diagram labeled (4) is isotopic to the diagram of \( L ' \) in \Cref{4Fig:links1}.
\begin{figure}[H]
\includegraphics[scale=0.45]{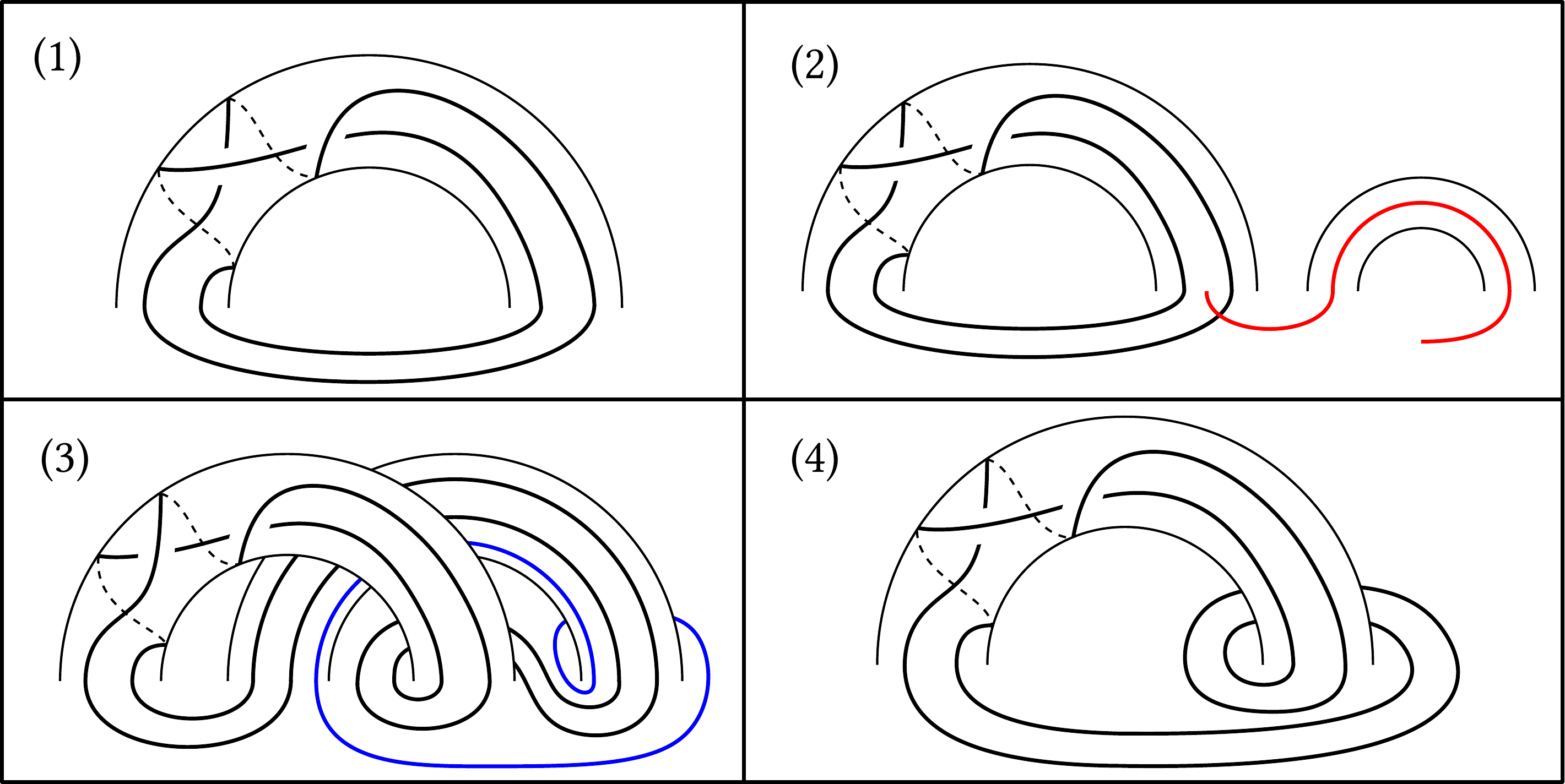}
\caption{Realizing a Dehn twist as a concordance.}
\label{4Fig:movie}
\end{figure}

As \( L \) and \( L' \) are both links in \( \Sigma_1 \times I \) we must obstruct their pseudostrict concordance in order to apply \Cref{4Thm:ascent}. At first glance the fact that \( L \) and \( L ' \) are related by a Dehn twist may lead one to attempt to show that they are strictly concordant, by making a particular choice of identification of the boundary of the target \(3 \)-manifold. However, as we are considering links in thickened surfaces up to isotopy only, such attempts fail.

For example, let \( D \looparrowright \Sigma_1 \) be the diagram of \( L \) given in \Cref{4Fig:links1}, and consider the cobordism \( ( S , M ) \) where \( S = D \times I \) and \( M = \Sigma_1 \times I \). Identify the boundary of \( M \) via \( \Psi : \partial M \longrightarrow \Sigma_1 \sqcup \Sigma_1 \) such that \( \Psi |_{\Sigma_1 \times \lbrace 0 \rbrace} \) is the identity and \( \Psi |_{\Sigma_1 \times \lbrace 1 \rbrace} = \psi \).

We have \( \partial S = D \sqcup \widetilde{D} \), where \( \widetilde{D} = \psi ( D ) \) is a diagram on \( \Sigma_1 \). As described in \Cref{2Subsec:lits}, two diagrams on surfaces may be compared only if the ambient surfaces are identical. This is a consequence of the fact that we consider links in thickened surfaces up to isotopy only. The diagrams \( D \) and \( \widetilde{D} \) appearing on the boundary of \( S \) do not have identical ambient space: the ambient space of \( \widetilde{D} \) is obtained from that of \( D \) via \( \psi \). It follows that, in order to compare \( D \) and \( \widetilde{D} \), we must apply \( \psi^{-1} \) to \( \widetilde{D} \). But \( \psi^{-1} ( \widetilde{D} ) = D \), so that \( ( S , M ) \) is a concordance from \( L \) to itself (as anticipated by the fact that \( S \) is a product cobordism). This argument applies \emph{mutatis mutandis} to other diagrams of \( L \) and to other boundary identifications.

We now obstruct the pseudostrict concordance of \( L \) and \( L' \). Notice that without loss of generality we may assume that a pseudostrict concordance does not contain index \( 0 \) or \( 3 \) critical points.
\begin{lemma}\label{4Prop:pseudostrict}
If there exists a pseudostrict concordance between two links, then there exist a pseudostrict concordance between them with only index \( 1 \) and \( 2 \) critical points.
\end{lemma}

\begin{proof}
Let \( ( S, M ) \) be a pseudostrict concordance. Traverse \( ( S , M ) \) from the initial link, until an index \( 3 \) or index \( 0 \) critical point is met. Suppose it is of index \( 3 \), and that \( \Sigma_g \) is the level surface preceding it: when passing such a critical point the level surface becomes \( S^2 \sqcup \Sigma_g \). As \( ( S, M ) \) is pseudostrict, in traversing the remainder of the concordance the \( S^2 \) component may only be attached to another level surface component, split into two disjoint copies of \( S^2 \), or remain unchanged. It follows that any regions of the cobordism surface \( S \) which require the presence of the index \( 3 \) critical point may equivalently be supported on (thickenings of ) disc neighbourhoods of \( \Sigma_g \). Produce a new pseudostrict concordance by reproducing such regions of \( S \) on a (thickenings of) disc neighbourhoods of \( \Sigma_g \), and removing the index \( 3 \) critical point (and any subsequent critical points interacting with it). Any remaining index \( 3 \) critical points may be removed by repeating this process. We may remove index \( 0 \) critical points by considering the reverse cobordism and repeating the above argument.
\end{proof}

Both \( L \) and \( L ' \) have \( \Sigma_1 \times I \) as ambient space; suppose a pseudostrict concordance between them, \( ( S , M ) \), contains an index \( 1 \) critical point which necessarily creates a disjoint \( S^2 \) component. As \( ( S , M ) \) is pseudostrict any regions of \( S \) which require the presence of this \( S^2 \) may be reproduced in (thickenings of) disc neighbourhoods of \( \Sigma_1 \), and by an argument similar to that given in the proof of \Cref{4Prop:pseudostrict} we may produce a new pseudostrict concordance without index \( 1 \) critical points. As this new concordance does not possess index \(0\), \(1\), or \(3\) critical points and \( \Sigma_1 \) is connected, it follows that it does not possess index \(2\) critical points between two distinct components. We conclude that if \( L \) and \( L ' \) are pseudostrictly concordant, then they are strictly concordant.

The links \( L \) and \( L ' \) are not strictly concordant as they are not homotopic in \( \Sigma_1 \times I \). This conclusion also makes use of the fact that the ambient spaces of the diagrams in \Cref{4Fig:links1} are identical. It follows that \( L \) and \( L ' \) are not pseudostrictly concordant, and by \Cref{4Thm:ascent} they are ascent concordant. One may produce an infinite family of ascent concordant links by iterating the Dehn twist.

\subsubsection{An elementary argument}\label{Sec:elem}
Although we used totally reduced homology to show that \( L \) and \( L ' \) are ascent concordant, there is an elementary method of doing so. We describe this method now, before presenting a pair of ascent links to which it cannot be applied.

Assume towards a contradiction that \( ( S , M ) \) is a descent concordance from \( L \) to \( L ' \) that is not pseudostrict. We may therefore decompose \( M \) as \( M = M_1 \cup_{F} M_2 \), where \( F \) is a disjoint union of \(2\)-spheres. The inclusion \( i : \Sigma_1 \rightarrow M_1 \) induces \( i_{\ast} : H_1 ( \Sigma_1 ; \mathbb{R} ) \rightarrow H_1 ( M_1 ; \mathbb{R} ) \) with \( \text{null} ( i_{\ast} ) = 1 \). Let \( S_1 = M_1 \cap S \); if \( \partial S_1 = L \sqcup \widetilde{L} \), then \( \widetilde{L} \hookrightarrow F \times I \) and \( [ \widetilde{L} ] = 0 \in H_1 ( F ; \mathbb{R} ) \) so that \( [ \widetilde{L} ] = 0 \in H_1 ( M ; \mathbb{R} ) \) also. As \( \partial S_1 = L \sqcup \widetilde{L} \), it follows that \( [ L ] = [ \widetilde{L} ] = 0 \in H_1 ( M ; \mathbb{R} ) \). But the components of \( L \) span a rank \( 2 \) subspace of \( H_1 ( \Sigma_1 ; \mathbb{R} ) \), so that \( 2 \leq \text{null} ( i_{\ast} ) \leq 1 \), a contradiction. As we have obstructed the pseudostrict concordance of \( L \) and \( L ' \), the argument above is enough to show that they are ascent concordant.

\subsubsection{Second family}\label{Sec:2fam}
\begin{figure}
\includegraphics[scale=0.6]{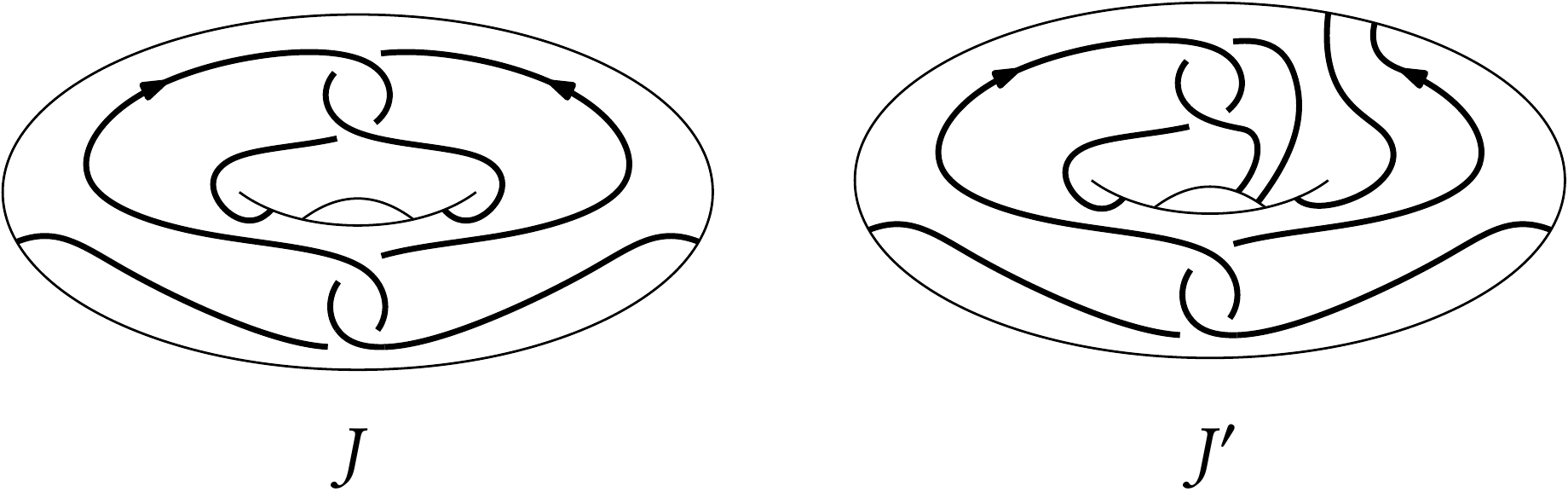}
\caption{A pair of ascent concordant links not amenable to the elementary method.}
\label{4Fig:links2}
\end{figure}

We now exhibit a pair of ascent concordance links to which the method of \Cref{Sec:elem} cannot be applied. Consider the links \( J \) and \( J ' \) depicted in \Cref{4Fig:links2}. The concordance described in \Cref{4Fig:movie} can be easily modified to show that \( J \) and \( J' \) are concordant.

The link \( J \) is totally nontrivial: \( \dkh '' ( J , \alpha ) \) and \( \dkh '' ( J , \beta ) \) have generators of \( ( j , c ) \)-bidegree \( ( -2, 1 ) \), and \( \dkh '' ( J , \alpha + \beta ) \) has a generator of bidegree \( ( -8, -2 ) \). The pseudostrict concordance of \( J \) and \( J ' \) may be obstructed exactly as in the case of \( L \) and \( L ' \) above.

Notice that the components of \( J \) span a rank \( 1 \) subspace of \( H_1 ( \Sigma_1 ; \mathbb{R} ) \), so that the method given in \Cref{Sec:elem} does not obstruct their descent concordance. However, \Cref{4Thm:ascent} may be applied to obstruct their descent concordance, so that \( J \) and \( J ' \) are ascent concordant. As with \( L \) and \( L ' \), one may produce an infinite family of ascent concordant links by iterating the Dehn twist on \( J \).

That \( J \) and \( J ' \) are ascent concordant reveals a subtly regarding concordance depicted in \Cref{4Fig:movie}\label{4Page:handles}; denote it by \( (S,M) \). The \(3\)-manifold \( M \) is equal to \( \Sigma_1 \times I \) in a non-minimal handle decomposition. To see this, notice that the the attaching sphere of the \( 2 \)-handle (the blue curve in the third panel) intersects the belt sphere of the \(1\)-handle added in second panel exactly once. These handles therefore form a cancelling pair. In turn, the induced decomposition of \( M \times I \) contains a cancelling pair of handles. In the cobordism described in \Cref{4Fig:movie}, \( S \) is a disjoint union of two annuli. The result that \( J \) and \( J ' \) are ascent concordant implies that the cancelling pair of handles of \( M \times I \) cannot be cancelled in the complement of (a neighbourhood) of \( S \). If they could, this would yield a new concordance from \( J \) to \( J ' \), \( (S', M')\), in which \( M ' \) is identically equal to \( \Sigma_1 \times I \). Such a concordance would not contain an exceeding critical point, contradicting the fact that \( J \) and \( J' \) are ascent concordant. Thus the non-minimal handle decomposition of \( M \times I \) cannot be simplified once \( S \) has been embedded.

In summary, a pair of cancelling handles are added to \( M = \Sigma_1 \times I \), and the disjoint union of annuli \( S \) is embedded into \( M \times I \). As \( J \) and \( J'\) are ascent concordant, the cancelling handles of \( M \times I\) cannot be cancelled in the complement of \( S \). The totally reduced homology is used to prove the ascent concordance of \( J \) and \( J '\); in light of the discussion above, we may interpret this result as an application of link homology to the study of knotted surfaces. In particular, the result demonstrates that the totally reduced homology is able to detect the subtle fact that a non-minimal handle decomposition of \(M \times I \) is required to realise a concordance from \( J \) to \( J ' \). This is evidence that the totally reduced homology contains interesting information regarding the topology of the complements of knotted surfaces.

\subsubsection{More examples}\label{Sec:evenmore}
\begin{figure}
\includegraphics[scale=0.6]{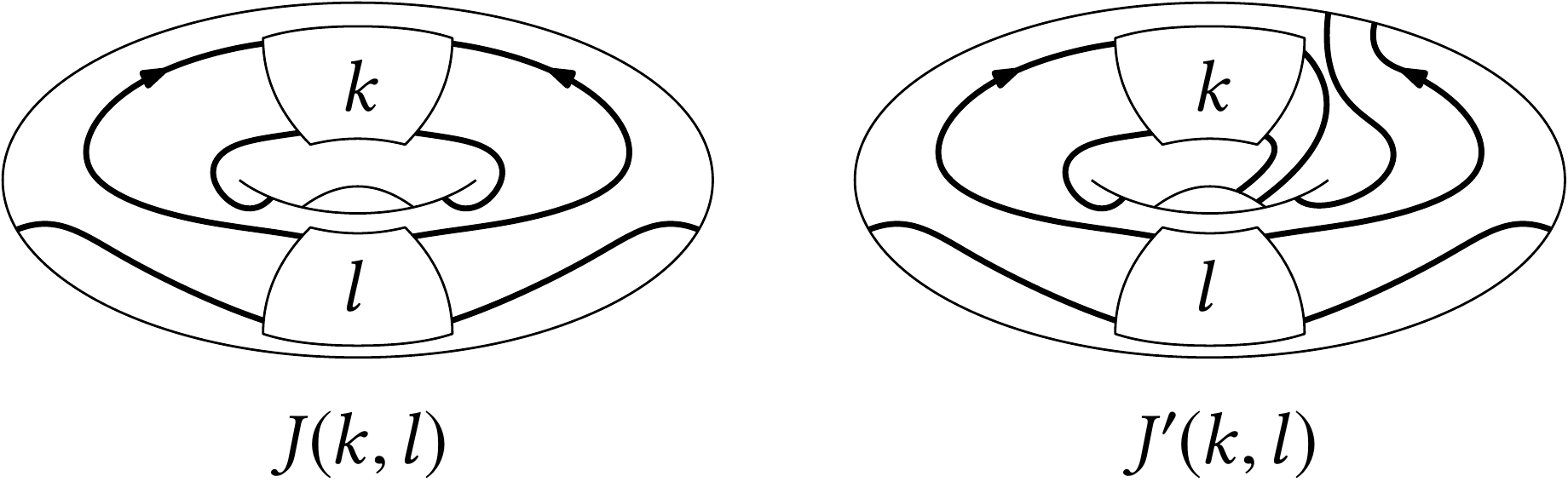}
\caption{The links \( J ( k ,l ) \) and \( J ' ( k ,l ) \). A box labelled \( i \) denotes \( i \) full twists.}
\label{4Fig:links3}
\end{figure}

We conclude by presenting more examples of ascent concordant links, that are not amenable to the method of \Cref{Sec:elem} (or an appropriate generalization of it). First, consider the links depicted in \Cref{4Fig:links3}: a box labelled \( i \) denotes \( i \in \Z \) full twists so that \( J ( -1 , -1 ) = J \), for \( J \) given in \Cref{4Fig:links2}. Suppose that \( k , l > 0 \). In this case \( \dkh '' ( J ( k ,l ) , \alpha ) \) and \( \dkh '' ( J ( k ,l ) , \alpha + \beta ) \) both have a generator of \( ( j , c ) \)-bidegree \( ( 0, 4-k-l ) \), and \( \dkh '' ( J ( k ,l ) , \beta ) \) has a generator of bidegree \( ( k + l - 2, 1 + \frac{1}{2} ( k + l ) \). Thus \( J ( k ,l ) \) is totally nontrivial for \( k + l \neq 4 \).

It follows that \Cref{4Thm:ascent} and the discussion in \Cref{Sec:1fam,Sec:2fam} also apply to prove that \( J ( k ,l ) \) and \( J ' ( k ,l ) \) are ascent concordant. Notice that the components of \( J ( k ,l ) \) span a rank \( 1 \) subspace of \( H_1 ( \Sigma_1 ; \mathbb{R} ) \) so that the method of \Cref{Sec:elem} does not apply.

\begin{figure}
\includegraphics[scale=0.625]{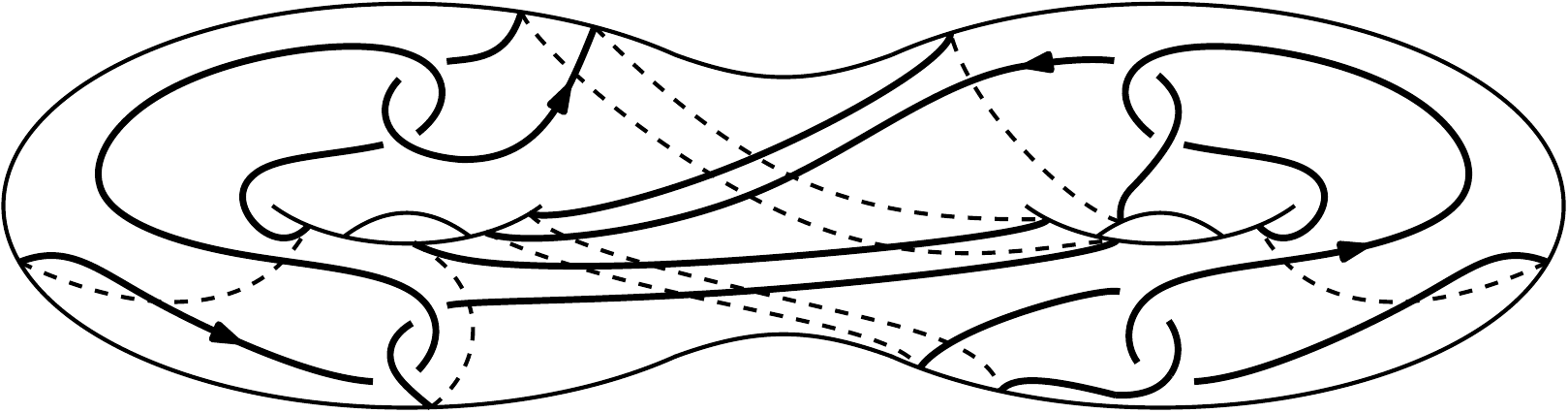}
\caption{A totally nontrivial link on a genus \(2\) surface.}
\label{4Fig:links4}
\end{figure}

Finally, consider the link depicted in \Cref{4Fig:links4}. This link is totally nontrivial so that performing Dehn twists, and appealing to \Cref{4Thm:ascent} and the discussion in \Cref{Sec:1fam,Sec:2fam}, one can produce a number of pairs of ascent concordant links. The components of this link span a rank \(2\) subspace of \( H_1 ( \Sigma_2 ; \mathbb{R} ) \), so that an appropriate generalization of the method given in \Cref{Sec:elem} cannot be used to obtain these ascent concordant pairs.

\bibliographystyle{plain}
\bibliography{library}

\end{document}